\documentclass{article}
\usepackage{amssymb,amsmath,amsfonts,amsthm,graphicx,ifthen,mathrsfs,hyperref}
\usepackage{pgf,tikz}
\usepackage{pgfplots}
\pgfplotsset{compat=1.15}
\usepackage{mathrsfs}
\usetikzlibrary{arrows}
\usepackage{float} 
\usepackage{subcaption}

\usepackage[cmtip,arrow]{xy}
\usepackage{comment}
\usepackage{mathtools}
\usepackage{ifthen}

\theoremstyle{plain}
\newtheorem{Theorem}{Theorem}   [section]
\newtheorem{Lemma}[Theorem]   {Lemma}

\newtheorem{Ruleset}[Theorem] {Ruleset}

\newtheorem{Open}{Open}

\newcommand{\craig}[1]{\textcolor{purple}{#1}}
\newcommand{\kyle}[1]{\textcolor{green}{#1}}

\newcommand{\Pclass}{\mathcal{P}}
\newcommand{\Nclass}{\mathcal{N}}
\newcommand{\fc}{\textsc{flag coloring}}

\newcommand{\at}{\textsc{avoid true}}

\newcommand{\ourmod}{\text{ mod }}
\newcommand{\cclass}[1]{\ensuremath{\mathord{\rm #1}}}

\newcommand{\outcomeClass}[1]{\ensuremath{\mathcal{#1}}}
\newcommand{\outP}{\outcomeClass{P}}
\newcommand{\outN}{\outcomeClass{N}}
\newcommand{\impartialGameSet}[1]{\ensuremath{{}^{*\!}\{\ #1\ \} }}

\begin{document}

\title{Vexing vexillological logic}
\author{Kyle Burke $^1$, Craig Tennenhouse $^2$}
\date{
$^1$Florida Southern College, Lakeland, FL 33801, USA\\{\tt kburke@flsouthern.edu}\\
$^2$University of New England, Biddeford, ME 04005, USA\\{\tt ctennenhouse@une.edu}
}

\maketitle

\begin{abstract}
We define a new impartial combinatorial game, \fc,  based on flood filling. We then generalize to a graph game, and find values for many positions on two colors.  We demonstrate that the generalized game is PSPACE-complete for two colors or more via a reduction from \at, determine the outcome classes of games based on real-world flags, and discuss remaining open problems.
\end{abstract}

\section{Introduction}\label{sec:intro}

In Sect. \ref{sec:intro} we present some background and definitions in Combinatorial Game Theory. We then introduce a new impartial game defined on a colored plane region (like a flag) and extend it to graphs, determining results on some families of graphs in Sect. \ref{sec:results}. In Sect. \ref{sec:flags} we examine optimal play on flags of some real-world regions. We prove in Sect. \ref{sec:complexity} that the game played on general graphs is in PSPACE. The paper concludes in Sect. \ref{sec:conclusion} with a list of unanswered questions about the game.

\emph{Combinatorial Game Theory} is the study of two-person strategy games with perfect information. A \emph{ruleset} is the description of play, including permitted moves, and a game played under \emph{normal play} is won by the last player to make a permitted move. A \emph{game} or \emph{position} is one state, and the \emph{options} of a game are those positions that are reachable by the next player. 

A combinatorial game is \emph{impartial} if both players have access to the same options at every position. Following the method of Grundy \cite{Grundy:1939} and Sprague \cite{Sprague:1936}, each position is assigned a non-negative integer value equal to the minimum non-negative integer value excluded (\emph{mex}) from the values of its options. This integer is called a \emph{nimber}, written $\ast n$, and any game with no options is assigned the nimber $\ast 0$ or simply $0$. In this paper we use the notation $\impartialGameSet{\ast n_1, \ast n_2, \ldots ,\ast n_k}$ to denote the value of a position with option values $\ast n_1, \ast n_2, \ldots ,\ast n_k$. Thus, $\impartialGameSet{\ast n_1, \ast n_2, \ldots, \ast n_k} = \ast \left(\text{mex}(n_1, n_2, \ldots, n_k) \right)$.  
A position is in \outP\ if it has value $0$ and in \outN\ otherwise. Positions in \outP\ do not have a winning strategy for the next player, while positions in \outN\ do.  By convention, we will use a game position $G$ with value $\ast n$ and the position $G$ itself interchangeably, so $G = \ast n$. For other conventions and further background see \cite{LessonsInPlay:2007, WinningWays:2001}.

There have been a number of mathematical games and puzzles based on the mechanism of a \emph{flood fill} algorithm, i.e. wherein a player or players change the color of a contiguous region in the plane. \textsc{The Honey-Bee Game} is one such game, a partisan game played on a hexagonal board, and whose computational complexity is analyzed in \cite{fleischer2012algorithmic}. 
\textsc{Flood-It} is a single player flood-filling 
puzzle analyzed in \cite{vu2019extremal} 
and \cite{meeks2013complexity}
. Further analysis of the complexity of \textsc{Flood-It} can be found in \cite{fellows2018survey}
. Our ruleset differs from these in that it is a two-player impartial game.

\begin{Ruleset}
\fc\ is a combinatorial game played on a colored simple graph $G$. A turn consists of choosing a vertex $v$ of color $c$ and a color $c'\neq c$, the color of a neighboring vertex. The vertex $v$ and all vertices in the same connected color $c$ component are then colored $c'$. 
\end{Ruleset}

We study playing \fc\ under the Normal Play convention, so that when no moves remain the game is over and the last player to move wins. We consider the cases where $G$ is connected, and thus the game ends when only one color remains. In practice, two-color \fc~can be played on a colored two-dimensional field, wherein players take turns reducing the number of distinctly-colored regions of the field. Each initial colored region can be mapped to a single vertex of a graph $G$, with neighboring regions associated with adjacent vertices in $G$.

In this manuscript we address the generalized game, played on a simple graph. Note that a single move can be realized by a graph with fewer vertices than the initial graph, wherein any edge adjoining two vertices of the same color is contracted and the process is repeated until the graph is properly colored. Unless stated otherwise, assume that any graph $G$ associated with a game of \fc~is such a properly colored graph. Similarly, unless otherwise noted, assume that every game is played on a properly two-colored bipartite graph. By $u\sim v$ we mean that the vertices $u$ and $v$ are adjacent.

\section{Results}
\label{sec:results}

In this section, we investigate many graph families and find the nimber values for two-color \fc\ positions on these graphs.

Let $S_i$ be the star graph with one central vertex and $i$ vertices adjacent only to that center. Note that there is no consensus on this notation, and some other sources denote this graph by $S_{i+1}$.

\begin{Theorem}\label{thm:stars}
\fc\ on $S_i = 
  \begin{cases}
    0, & \text{if } i = 0\\
    \ast, & \text{if } i \mod 2 = 1\\
    \ast 2, & \text{if } i \geq 2 \text{ and } i \mod 2 = 0
  \end{cases}$
\end{Theorem}
\begin{proof}
Each position has two functionally-different options: a move to zero by playing on the center or to $S_{i-1}$ by playing on one of the leaves.  Thus, only the edgeless case ($S_0$) has no options and is zero; the other positions alternate $\ast$ and $\ast 2$ as $i$ increases.
\end{proof}

Let $T_{p,d}$ be the two-colored graph consisting of a single vertex $v$ adjacent to $p$-many pendants, and with $d$-many triples of vertices $\{u_i, w_i, x_i\}$, and the edges $(u_i, x_i), (w_i,x_i), (u_i, v), $ and $(w_i, v)$ (Fig.~\ref{fig:diamonds}).

\begin{figure}[h!]
\centering
    \begin{tikzpicture}[node distance = 1cm, every node/.style={draw=black,circle}]

      \node[draw] (v) [] {$v$};
      \node[draw, fill=black] (p1) [left of=v, label=left:{$p_1$}] {};
      \node[draw, fill=black] (p2) [below left of=v, label=left:{$p_2$}] {};
      \node[draw, fill=black] (p3) [below of=v, label=below:{$p_3$}] {};
      \node[draw, fill=black] (w1) [above right of=v, label=right:{$w_1$}] {};
      \node[draw, fill=black] (u1) [above of=v, label=left:{$u_1$}] {};
      \node[draw] (x1) [above right of=u1, label=right:{$x_1$}] {};      
      \node[draw, fill=black] (w2) [below right of=v, label=below right:{$w_2$}] {};
      \node[draw, fill=black] (u2) [right of=v, label=right:{$u_2$}] {};
      \node[draw] (x2) [below right of=u2, label=right:{$x_2$}] {};

      \path[-]
        (v) edge (p1)
        (v) edge (p2)
        (v) edge (p3)
        (v) edge (u1)
        (v) edge (w1)
        (u1) edge (x1)
        (w1) edge (x1)
        (v) edge (u2)
        (v) edge (w2)
        (u2) edge (x2)
        (w2) edge (x2)        
        
      ;
    \end{tikzpicture}
\caption{The \fc~ position $T_{3,2}$; three pendants and two diamonds.}
\label{fig:diamonds}
\end{figure}
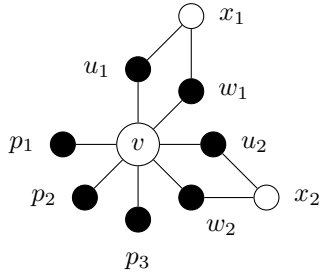

\begin{Theorem}\label{thm:diamonds}
The value of a \fc\ position on $T_{p,d}$ is given by Table \ref{table:Ts}.
\end{Theorem}

\begin{table}[h!]
  \begin{center}
  \begin{tabular}{|r||rrrrr||rr|}
    \hline
    $p\backslash d$ & 0 & 1 & 2 & 3 & 4  & $2k$ & $2k + 1$ \\ \hline \hline
      0 & 0 & 0 & 0 & 0 & 0  & 0 & 0 \\ 
      1 & $\ast$ & $\ast 3$ & $\ast$ & $\ast 2$ & $\ast$  & $\ast$ & $\ast 2$ \\ 
      2 & $\ast 2$ & 0 & 0 & 0 & 0  & 0 & 0 \\ 
      3 & $\ast$ & $\ast 3$ & $\ast$ & $\ast 2$ & $\ast$  & $\ast$ & $\ast 2$ \\ 
      4 & $\ast 2$ & 0 & 0 & 0 & 0  & 0 & 0 \\ 
      \hline \hline 
      $2l$ & $\ast 2$ & 0 & 0 & 0 & 0  & 0 & 0 \\ 
      $2l+1$ & $\ast$ & $\ast 3$ & $\ast$ & $\ast 2$ & $\ast$  & $\ast$ & $\ast 2$ \\ \hline
  \end{tabular}
  \end{center}
  \caption{Values of $T_{p,d}$ positions.}
  \label{table:Ts}
\end{table}

\begin{proof}
The $d=0$ column is exactly the $S_p$ case from Theorem \ref{thm:stars}.  The $p=0$ row is all zeroes because a play on the center results in $S_d$, which is in \outN, and any play on a diamond results in replacing that diamond with a pendant, which can be played by the opponent, resulting in another $p=0$ case.  In general, when $d > 0$ and $p > 0$, the options are: $S_d$ (by playing at the center), $T_{p-1, d}$ (by playing on any one of the pendants), and $T_{p+1, d-1}$ (by playing on any of the three vertices of any diamond).  This means that for any cell in the table, the options are to the cell above ($T_{p-1, d}$), the cell directly below and to the left ($T_{p+1, d-1}$), and the $S_d$ cell along the left edge of the table, which will either be $\ast$ or $\ast 2$.  

For the $d=1$ column, each cell has a move to $T_{1,0} = S_1 = \ast$.  The $p=0$ cell will be zero, because both possible moves are to $T_{1,0} = \ast$.  Then, inductively, the odd-$p$-valued cells have additional options above to 0 and below-and-left to $\ast 2$, so their value will be $\ast 3$.  The even-$p$-valued cells have options above to $\ast 3$ and below-and-left to $\ast$, so they have value zero.

For any $T_{p, d}$ where $p > 0, d > 1$, we can complete the proof by strong induction.  Assume that all cells above $p,d$ are as shown in Table \ref{table:Ts}, as well as the entire columns for lower $d$.  Then, given the three options spelled out above, $T_{p, d} = \impartialGameSet{S_d, T_{p-1, d}, T_{p+1, d-1}}$.  In the case where $p$ is even, these are $\impartialGameSet{ (\ast \text{ or } \ast 2), (\ast, \ast 2, \text{ or } \ast 3), (\ast, \ast 2, \text{ or } \ast 3)} = 0$.  When $p$ is odd and $d$ is even, these are $\impartialGameSet{\ast 2, 0, 0} = \ast$.  When $p$ and $d$ are both odd, these are $\impartialGameSet{\ast, 0, 0} = \ast 2$, completing the cases of the proof.
\end{proof}

\begin{Theorem}\label{thm:paths}
A path of length $x, (P_{x+1})$, has value $\ast (x \pmod 3)$. 
\end{Theorem}
\begin{proof}
(Proof by induction on $x$.)  Base cases: $x \in \{0, 1\}$.  For $x=0$, there are no moves, so the value is 0.  For $x = 1$, there is only one option to 0, so the value is $\ast$.  

For the inductive case ($x \geq 2$) a play on one of the endpoints results in $P_{x}$; a play on any of the other vertices collapses that vertex with both neighbors, resulting in $P_{x-1}$.  By the inductive hypothesis, we can calculate the value based on these two options:

$$
\begin{cases}
  \{P_{0\ourmod 3}, P_{2 \ourmod 3}\} = \{\ast 2, \ast \} = 0 &, \text{if } x \ourmod 3 = 0 \\
  \{P_{1\ourmod 3}, P_{0 \ourmod 3}\} = \{0, \ast 2\} = \ast &, \text{if } x \ourmod 3 = 1 \\
  \{P_{2\ourmod 3}, P_{1 \ourmod 3}\} = \{\ast, 0\} = \ast 2 &, \text{if } x \ourmod 3 = 2
\end{cases}
$$
\end{proof}

\begin{Theorem}\label{thm:pendantsPath}
Let a \emph{broom graph} $B_{i,l}$ be the graph consisting of a path of length $l$ with end vertices $u$ and $v$, along with vertices $\{v_1,\ldots ,v_{i}\},i\geq 0,$ adjacent to $v$. Then the $2$-color \fc\ game on $B_{i,l}$ has the following values:

\begin{center}
\begin{tabular}{|r||rrrrrr|}
\hline
  $i$\textbackslash$l$  & 0 & 1 & 2 & $3j + 3$ & $3j + 4$ & $3j + 5$ \\ \hline \hline
  0        & 0 & $\ast$ & $\ast 2$ & 0 & $\ast$ & $\ast 2$ \\ 
  1        & $\ast$ & $\ast 2$ & 0 & $\ast$ & $\ast 2$ & 0 \\ 
  2        & $\ast 2$ & $\ast$ & $\ast 3$ & 0 & $\ast$ & $\ast 2$ \\ 
  $2k + 3$ & $\ast$ & $\ast 2$ & 0 & $\ast$ & $\ast 2$ & 0 \\ 
  $2k + 4$ & $\ast 2$ & $\ast$ & $\ast 2$ & 0 & $\ast$ & $\ast 2$ \\ \hline
\end{tabular}
\end{center}

\end{Theorem}
\begin{proof}
We proceed by induction on $i$. Our base cases are where $i \in \{0, 1, 2\}$; cases 0 and 1 are just paths and are covered by Theorem \ref{thm:paths}.  For the case where $i=2$, we will again use induction on $l$.  The base cases are $l \leq 4$.  For $l=0$, the graph is equivalent to $P_3 = \ast 2$ by Theorem \ref{thm:paths}.  For $l=1$, the graph is a star.  Playing on any pendant vertex adjacent to $v$ moves to $P_3$ ($\ast 2$) and playing on the center moves to 0, so the value is $\ast$.   The remaining base cases for $l \in \{2, 3, 4\}$ are covered in parts a, b, and c of Figure \ref{fig:pendantsPathl2}.  The options for the inductive cases for $l \geq 5$ are drawn out in part d of the same figure.  The options $B_{2,l-1}$ and $P_{l-1}$ both have value $\ast ((l-1) \ourmod 3)$; the options $B_{2, l-2}$ and $P_{l+1}$ have value $\ast ((l+1) \ourmod 3)$.  $\{\ast ((l-1) \ourmod 3), \ast ((l+1) \ourmod 3)\} = \ast (l \ourmod 3)$.  This completes the base cases on $i$.

\begin{figure}[h!]
  \begin{subfigure}[b]{.45\textwidth}
    \begin{center}
      \begin{tikzpicture}[node distance = 1cm, every node/.style={draw=black,circle}]
        \node[draw] (p1) [label=right:{0}] {};
        \node[draw, fill=black] (v) [below left of=p1, label=below:{$\ast$}] {};
        \node[draw] (p2) [below right of=v, label=right:{0}] {};
        \node[draw] (c1) [left of=v, label=below:{$\ast 2$}] {};
        \node[draw, fill=black] (c2) [left of=c1, label=below:{$\ast$}] {};

        \path[-]
          (p1) edge (v)
          (v) edge (p2)
          (v) edge (c1)
          (c1) edge (c2)
        ;
      \end{tikzpicture}
      \caption{$l=2$: value $= \ast 3$}
    \end{center}
  \end{subfigure}
  \begin{subfigure}[b]{.45\textwidth}
    \begin{center}
      \begin{tikzpicture}[node distance = 1cm, every node/.style={draw=black,circle}]
        \node[draw] (p1) [label=right:{$\ast$}] {};
        \node[draw, fill=black] (v) [below left of=p1, label=below:{$\ast 2$}] {};
        \node[draw] (p2) [below right of=v, label=right:{$\ast$}] {};
        \node[draw] (c1) [left of=v, label=below:{$\ast$}] {};
        \node[draw, fill=black] (c2) [left of=c1, label=below:{$\ast$}] {};
        \node[draw] (c3) [left of=c2, label=below:{$\ast 3$}] {};

        \path[-]
          (p1) edge (v)
          (v) edge (p2)
          (v) edge (c1)
          (c1) edge (c2)
          (c2) edge (c3)
        ;
      \end{tikzpicture}
      \caption{$l=3$: value $= 0$}
    \end{center}
  \end{subfigure}
  
  \begin{subfigure}[b]{.45\textwidth}
    \begin{center}
      \begin{tikzpicture}[node distance = 1cm, every node/.style={draw=black,circle}]
        \node[draw] (p1) {};
        \node[draw, fill=black] (v) [below left of=p1, label=below:{0}] {};
        \node[draw] (p2) [below right of=v, label=below:{$\ast 2$}] {};
        \node[draw] (c1) [left of=v, label=below:{$\ast 3$}] {};
        \node[draw, fill=black] (c2) [left of=c1, label=below:{$\cdots$}] {};
        \node[draw] (c3) [left of=c2, label=below:{$\ast 3$}] {};
        \node[draw, fill=black] (c4) [left of=c3, label=below:{0}] {};

        \path[-]
          (p1) edge (v)
          (v) edge (p2)
          (v) edge (c1)
          (c1) edge (c2)
          (c2) edge (c3)
          (c3) edge (c4)
        ;
      \end{tikzpicture}
      \caption{$l=4$: value $= \ast$}
    \end{center}
  \end{subfigure}
  \begin{subfigure}[b]{.45\textwidth}
    \begin{center}
      \begin{tikzpicture}[node distance = 1cm, every node/.style={draw=black,circle}]
        \node[draw] (p1) {};
        \node[draw, fill=black] (v) [below left of=p1, label=below:{$P_{l-1}$}] {};
        \node[draw] (p2) [below right of=v, label=below:{$P_{l+1}$}] {};
        \node[draw] (c1) [left of=v] {};
        \node[draw=none, rectangle] (dots) [left of=c1, label=below:{$B_{2, l-2}$}] {$\cdots$};
        \node[draw, fill=gray] (c2) [left of=dots] {};
        \node[draw, fill=gray] (c3) [left of=c2, label=below:{$B_{2, l-1}$}] {};

        \path[-]
          (p1) edge (v)
          (v) edge (p2)
          (v) edge (c1)
          (c1) edge (dots)
          (dots) edge (c2)
          (c2) edge (c3)
        ;
      \end{tikzpicture}
      \caption{$l\geq 5$: value $= \ast (l \text{ mod } 3)$}
    \end{center}
  \end{subfigure}
  \caption{Cases for Theorem \ref{thm:pendantsPath} where $i=2$.  Labels near vertices show resulting nimbers of playing at that vertex.  a, b, and c are base cases; d, covers the inductive cases.}
  \label{fig:pendantsPathl2}
\end{figure}

For the inductive cases $i \geq 3$, we split into the odd and even cases: $i = 2k + 3$ and $i = 2k + 4$ ($k \geq 0$).  In Figure \ref{fig:pendants_path_l3} we look at the odd case and show the options for the base cases ($i \leq 2$; a, b, c) and the inductive cases ($l \geq 3$; d).  In the inductive case $B_{o, l}$, the options $B_{o, l-1}$ and $B_{e, l}$ are both equal to $\ast (l \ourmod 3)$, while the options $B_{o, l-2}$ and $P_{l-1}$ both have value $\ast ((l-1) \ourmod 3)$.  Thus, the value of $B_{o, l}$ is $\ast ((l+1) \ourmod 3)$.

\begin{figure}[h!]
  \begin{subfigure}[b]{.45\textwidth}
    \begin{center}
      \begin{tikzpicture}[node distance = 1cm, every node/.style={draw=black,circle}]
        \node[draw] (p1) [label=right:{$\ast 2$}] {};
        \node[draw, fill=black] (v) [below left of=p1, label=below:{0}] {};
        \node[draw = none, rectangle] (vdots) [right of=v] {$\vdots$};
        \node[draw] (p2) [below right of=v, label=right:{$\ast 2$}] {};

        \path[-]
          (p1) edge (v)
          (v) edge (vdots)
          (v) edge (p2)
        ;
      \end{tikzpicture}
      \caption{$l=0$: value $= \ast$}
    \end{center}
  \end{subfigure}
  \begin{subfigure}[b]{.45\textwidth}
    \begin{center}
      \begin{tikzpicture}[node distance = 1cm, every node/.style={draw=black,circle}]
        \node[draw] (p1) [label=right:{$\ast$}] {};
        \node[draw, fill=black] (v) [below left of=p1, label=below:{0}] {};
        \node[draw = none, rectangle] (vdots) [right of=v] {$\vdots$};
        \node[draw] (p2) [below right of=v, label=right:{$\ast$}] {};
        \node[draw] (c1) [left of=v, label=below:{$\ast$}] {};

        \path[-]
          (p1) edge (v)
          (v) edge (vdots)
          (v) edge (p2)
          (v) edge (c1)
        ;
      \end{tikzpicture}
      \caption{$l=1$: value $=\ast 2$}
    \end{center}
  \end{subfigure}
  
  \begin{subfigure}[b]{.45\textwidth}
    \begin{center}
      \begin{tikzpicture}[node distance = 1cm, every node/.style={draw=black,circle}]
        \node[draw] (p1) [label=right:{$\ast 3$}] {};
        \node[draw, fill=black] (v) [below left of=p1, label=below:{$\ast$}] {};
        \node[draw = none, rectangle] (vdots) [right of=v] {$\vdots$};
        \node[draw] (p2) [below right of=v, label=below:{$\ast 3$}] {};
        \node[draw] (c1) [left of=v, label=below:{$\ast$}] {};
        \node[draw, fill=black] (c2) [left of=c1, label=below:{$\ast 2$}] {};

        \path[-]
          (p1) edge (v)
          (v) edge (vdots)
          (v) edge (p2)
          (v) edge (c1)
          (c1) edge (c2)
        ;
      \end{tikzpicture}
      \caption{$l=2$: value $=0$}
    \end{center}
  \end{subfigure}
  \begin{subfigure}[b]{.45\textwidth}
    \begin{center}
      \begin{tikzpicture}[node distance = 1cm, every node/.style={draw=black,circle}]
        \node[draw] (p1) {};
        \node[draw, fill=black] (v) [below left of=p1, label=below:{$P_{l-1}$}] {};
        \node[draw = none, rectangle] (vdots) [right of=v] {$\vdots$};
        \node[draw] (p2) [below right of=v, label=below:{$B_{e, l}$}] {};
        \node[draw] (c1) [left of=v] {};
        \node[draw=none, rectangle] (dots) [left of=c1, label=below:{$B_{o, l-2}$}] {$\cdots$};
        \node[draw, fill=gray] (c2) [left of=dots] {};
        \node[draw, fill=gray] (c3) [left of=c2, label=below:{$B_{o, l-1}$}] {};

        \path[-]
          (p1) edge (v)
          (v) edge (vdots)
          (v) edge (p2)
          (v) edge (c1)
          (c1) edge (dots)
          (dots) edge (c2)
          (c2) edge (c3)
        ;
      \end{tikzpicture}
      \caption{$l\geq 3$: value $= \ast ((l+1) \ourmod 3)$}
    \end{center}
  \end{subfigure}
  \caption{Cases for Theorem \ref{thm:pendantsPath} where $i= 2k + 3$.  Labels near vertices show resulting nimbers of playing at that vertex.  a, b, and c are base cases; d, covers the inductive cases.}
  \label{fig:pendants_path_l3}
\end{figure}

In Figure \ref{fig:pendants_path_l4} we examine the even case and show the options for the base cases ($i \leq 2$; a, b, c) and the inductive cases ($l \geq 3$; d).  In the inductive case $B_{e, l}$, the options $B_{e, l-1}$ and $P_{l-1}$ are both equal to $\ast ((l-1) \ourmod 3)$, while the options $B_{e, l-2}$ and $B_{o, l}$ both have value $\ast ((l+1) \ourmod 3)$.  Thus, the value of $B_{e, l}$ is $\ast (l \ourmod 3)$.  This completes all the cases. 
\end{proof}

\begin{figure}[h!]
  \begin{subfigure}[b]{.45\textwidth}
    \begin{center}
      \begin{tikzpicture}[node distance = 1cm, every node/.style={draw=black,circle}]
        \node[draw] (p1) [label=right:{$\ast$}] {};
        \node[draw, fill=black] (v) [below left of=p1, label=below:{0}] {};
        \node[draw = none, rectangle] (vdots) [right of=v] {$\vdots$};
        \node[draw] (p2) [below right of=v, label=right:{$\ast$}] {};

        \path[-]
          (p1) edge (v)
          (v) edge (vdots)
          (v) edge (p2)
        ;
      \end{tikzpicture}
      \caption{$l=0$: value $= \ast 2$}
    \end{center}
  \end{subfigure}
  \begin{subfigure}[b]{.45\textwidth}
    \begin{center}
      \begin{tikzpicture}[node distance = 1cm, every node/.style={draw=black,circle}]
        \node[draw] (p1) [label=right:{$\ast 2$}] {};
        \node[draw, fill=black] (v) [below left of=p1, label=below:{0}] {};
        \node[draw = none, rectangle] (vdots) [right of=v] {$\vdots$};
        \node[draw] (p2) [below right of=v, label=right:{$\ast 2$}] {};
        \node[draw] (c1) [left of=v, label=below:{$\ast 2$}] {};

        \path[-]
          (p1) edge (v)
          (v) edge (vdots)
          (v) edge (p2)
          (v) edge (c1)
        ;
      \end{tikzpicture}
      \caption{$l=1$: value $=\ast$}
    \end{center}
  \end{subfigure}
  
  \begin{subfigure}[b]{.45\textwidth}
    \begin{center}
      \begin{tikzpicture}[node distance = 1cm, every node/.style={draw=black,circle}]
        \node[draw] (p1) [] {};
        \node[draw, fill=black] (v) [below left of=p1, label=below:{$\ast$}] {};
        \node[draw = none, rectangle] (vdots) [right of=v] {$\vdots$};
        \node[draw] (p2) [below right of=v, label=below:{0}] {};
        \node[draw] (c1) [left of=v, label=below:{0}] {};
        \node[draw, fill=black] (c2) [left of=c1, label=below:{$\ast$}] {};

        \path[-]
          (p1) edge (v)
          (v) edge (vdots)
          (v) edge (p2)
          (v) edge (c1)
          (c1) edge (c2)
        ;
      \end{tikzpicture}
      \caption{$l=2$: value $=\ast 2$}
    \end{center}
  \end{subfigure}
  \begin{subfigure}[b]{.45\textwidth}
    \begin{center}
      \begin{tikzpicture}[node distance = 1cm, every node/.style={draw=black,circle}]
        \node[draw] (p1) {};
        \node[draw, fill=black] (v) [below left of=p1, label=below:{$P_{l-1}$}] {};
        \node[draw = none, rectangle] (vdots) [right of=v] {$\vdots$};
        \node[draw] (p2) [below right of=v, label=below:{$B_{o, l}$}] {};
        \node[draw] (c1) [left of=v] {};
        \node[draw=none, rectangle] (dots) [left of=c1, label=below:{$B_{e, l-2}$}] {$\cdots$};
        \node[draw, fill=gray] (c2) [left of=dots] {};
        \node[draw, fill=gray] (c3) [left of=c2, label=below:{$B_{e, l-1}$}] {};

        \path[-]
          (p1) edge (v)
          (v) edge (vdots)
          (v) edge (p2)
          (v) edge (c1)
          (c1) edge (dots)
          (dots) edge (c2)
          (c2) edge (c3)
        ;
      \end{tikzpicture}
      \caption{$l\geq 3$: value $= \ast (l \ourmod 3)$}
    \end{center}
  \end{subfigure}
  \caption{Cases for Theorem \ref{thm:pendantsPath} where $i= 2k + 4$.  Labels near vertices show resulting nimbers of playing at that vertex.  a, b, and c are base cases; d, covers the inductive cases.}
  \label{fig:pendants_path_l4}
\end{figure}




\begin{Theorem}\label{thm:2paths}
A properly two-colored complete bipartite graph is in $\Pclass$ if and only if each part has at least two vertices.
\end{Theorem}
\begin{proof}
A non-trivial star graph is in $\Nclass$, so consider the case where each part has at least two vertices. WLOG there is only one play. Once a node is colored, the result is equivalent to a properly two-colored star, which is in $\Nclass$.
\end{proof}

\newcommand{\Routes}[5]{%
\ifthenelse{\equal{#1}{0}}{%
  \ensuremath{P_{\ifthenelse{\equal{0}{#2}}{}{0^{#2}}
                 \ifthenelse{\equal{0}{#3}}{}{2^{#3}}
                 \ifthenelse{\equal{0}{#4}}{}{4^{#4}}
                 \ifthenelse{\equal{0}{#5}}{}{6^{#5}}}}}{%
  \ensuremath{P_{\ifthenelse{\equal{0}{#2}}{}{1^{#2}}
                 \ifthenelse{\equal{0}{#3}}{}{3^{#3}}
                 \ifthenelse{\equal{0}{#4}}{}{5^{#4}}}}}
}


Next we examine a number of positions consisting of two vertices $u$ and $v$ joined by a collection of internally disjoint paths. We will denote such a graph by $P_{l_1,l_2,\ldots}$ where $l_i$ is the length the $i^{th}$ path. WLOG we may assume that $l_1\leq l_2 \leq \ldots$. The following lemma will prove useful for a proof in Sec.~\ref{sec:complexity}.  

Alternative Notation: to simplify what would be long lists, we alternatively describe these positions with the subscripted lengths of the paths between $u$ and $v$, raised to the power of the number of times they appear.  Since these are proper two-colored graphs, the parity of all edges will be the same.  We use a path length of 0 to indicate that $u$ and $v$ have been merged into one vertex.  Thus, these positions fall into four categories: $\Routes{0}{1+}{i}{j}{k}$, even-length paths and $u = v$; $\Routes{1}{1+}{i}{j}{0}$, odd-length paths where $u$ and $v$ are neighbors; $\Routes{0}{0}{i}{j}{k}$, even-length paths where $u$ and $v$ are distinct; and finally $\Routes{1}{0}{i}{j}{0}$, odd-length paths where $u$ and $v$ are not neighbors.  We have nimber values for these shown in Tables \ref{table:Routes02461} through \ref{table:Routes35}.  (There's no table for $\Routes{0}{1+}{i}{j}{0}$, because those are equal to $T_{i, j}$, the values for which can be found in table \ref{table:Ts}.)

\begin{table}[h!]
  \begin{center}
  \begin{tabular}{|r||rrrr||cc|}
  \hline
    $i$\textbackslash$j$  & $0$ & $1$ & $2$ & $3$ & $2l + 4$ & $2l + 5$ \\ \hline \hline
    0        & $\ast  $ & $\ast  $ & $\ast  $ & $\ast  $ & $\ast  $ & $\ast  $\\ 
    1        & $\ast 2$ & $0     $ & $\ast 3$ & $0     $ & $\ast 3$ & $0     $ \\ 
    2        & $\ast  $ & $\ast  $ & $\ast  $ & $\ast  $ & $\ast  $ & $\ast  $\\ 
    3        & $\ast 2$ & $0     $ & $\ast 3$ & $0     $ & $\ast 3$ & $0     $ \\ 
    \hline \hline
    $2k + 4$ & $\ast  $ & $\ast  $ & $\ast  $ & $\ast  $ & $\ast  $ & $\ast  $\\ 
    $2k + 5$ & $\ast 2$ & $0     $ & $\ast 3$ & $0     $ & $\ast 3$ & $0     $ \\
    \hline
  \end{tabular}
  \end{center}
  \caption{Nimber values for $\protect \Routes{0}{1+}{i}{j}{1}$.}
  \label{table:Routes02461}
\end{table}

\begin{table}[h!]
  \begin{center}
  \begin{tabular}{|r||rrrr||cc|}
  \hline
    $i$\textbackslash$j$  & $0$ & $1$ & $2$ & $3$ & $2l + 4$ & $2l + 5$ \\ \hline \hline
    0        & $\ast 2$ & $0     $ & $\ast 2$ & $0     $ & $\ast 2$ & $0     $\\ 
    1        & $\ast  $ & $\ast 4$ & $\ast  $ & $\ast 4$ & $\ast  $ & $\ast 4$ \\ 
    2        & $\ast 2$ & $0     $ & $\ast 2$ & $0     $ & $\ast 2$ & $0     $\\ 
    3        & $\ast  $ & $\ast 4$ & $\ast  $ & $\ast 4$ & $\ast  $ & $\ast 4$ \\ 
    \hline \hline
    $2k + 4$ & $\ast 2$ & $0     $ & $\ast 2$ & $0     $ & $\ast 2$ & $0     $\\ 
    $2k + 5$ & $\ast  $ & $\ast 4$ & $\ast  $ & $\ast 4$ & $\ast  $ & $\ast 4$ \\
    \hline
  \end{tabular}
  \end{center}
  \caption{Nimber values for $\protect \Routes{0}{1+}{i}{j}{2}$.}
  \label{table:Routes02462}
\end{table}

\begin{table}[h!]
  \begin{center}
  \begin{tabular}{|r||rrrr||cc|}
  \hline
    $i$\textbackslash$j$  & $0$ & $1$ & $2$ & $3$ & $2l + 4$ & $2l + 5$ \\ \hline \hline
    0        & $\ast  $ & $0     $ & $\ast  $ & $0     $ & $\ast  $ & $0     $\\ 
    1        & $\ast 2$ & $\ast 3$ & $\ast 2$ & $\ast 3$ & $\ast 2$ & $\ast 3$ \\ 
    2        & $\ast  $ & $0     $ & $\ast  $ & $0     $ & $\ast  $ & $0     $\\ 
    3        & $\ast 2$ & $\ast 3$ & $\ast 2$ & $\ast 3$ & $\ast 2$ & $\ast 3$ \\ 
    \hline \hline
    $2k + 4$ & $\ast  $ & $0     $ & $\ast  $ & $0     $ & $\ast  $ & $0     $\\ 
    $2k + 5$ & $\ast 2$ & $\ast 3$ & $\ast 2$ & $\ast 3$ & $\ast 2$ & $\ast 3$ \\
    \hline
  \end{tabular}
 \end{center}
  \caption{Nimber values for $\protect \Routes{0}{1+}{i}{j}{3+}$.}
  \label{table:Routes02463}
\end{table}

\begin{table}[h!]
  \begin{center}
  \begin{tabular}{|r||rrrr||cc|}
  \hline
    $i$\textbackslash$j$  & $0$ & $1$ & $2$ & $3$ & $2l + 4$ & $2l + 5$ \\ \hline \hline
    0        & $\ast  $ & $\ast  $ & $\ast $ & $\ast$ & $\ast$ & $\ast$ \\ 
    1        & $0     $ & $0     $ & $0 $ & $ 0 $ & $ 0 $ & 0 \\ 
    2        & $\ast  $ & $\ast  $ & $\ast  $ & $\ast  $ & $\ast  $ & $\ast $ \\  
    3        & $0     $ & $0     $ & $0 $ & $ 0 $ & $ 0 $ & 0 \\ 
    \hline \hline
    $2k + 4$ & $\ast$ & $\ast$ & $\ast $ & $\ast$ & $\ast $ & $\ast$ \\ 
    $2k + 5$ & $ 0  $ & $ 0 $ & $0 $ & 0 & $0 $ & $0$ \\ \hline
  \end{tabular}
  \end{center}
  \caption{Nimber values for $\protect \Routes{1}{1+}{i}{j}{0}$.}
  \label{table:Routes135}
\end{table}

\begin{table}[h!]
  \begin{center}
  \begin{tabular}{|r||rrrr||cc|}
  \hline
    $i$\textbackslash$j$  & $0$ & $1$ & $2$ & $3$ & $2l + 4$ & $2l + 5$ \\ \hline \hline
    0        & $0     $ & $\ast  $ & $0     $ & $0     $ & $0     $ & $0    $ \\ 
    1        & $\ast 2$ & $\ast  $ & $\ast 2$ & $\ast  $ & $\ast 2$ & $\ast  $ \\  
    2        & $0     $ & $\ast  $ & $0     $ & $\ast  $ & $0     $ & $\ast  $ \\ 
    3        & $0     $ & $\ast  $ & $\ast 2$ & $\ast  $ & $\ast 2$ & $\ast  $ \\  
    \hline \hline
    $2k + 4$ & $0     $ & $\ast  $ & $0     $ & $\ast  $ & $0     $ & $\ast  $ \\ 
    $2k + 5$ & $0     $ & $\ast  $ & $\ast 2$ & $\ast  $ & $\ast 2$ & $\ast  $ \\   
   \hline
  \end{tabular}
  \end{center}
  \caption{Nimber values for $\protect \Routes{0}{0}{i}{j}{0}$.}
  \label{table:Routes24}
\end{table}

\begin{table}[h!]
  \begin{center}
  \begin{tabular}{|r||rrrr||cc|}
  \hline
    $i$\textbackslash$j$  & $0$ & $1$ & $2$ & $3$ & $2l + 4$ & $2l + 5$ \\ \hline \hline
    0        & $0     $ & $\ast  $ & $\ast  $ & $\ast  $ & $\ast  $ & $\ast  $ \\ 
    1        & $0     $ & $\ast 3$ & $0     $ & $\ast 3$ & $0     $ & $\ast 3$ \\  
    2        & $0     $ & $\ast  $ & $0     $ & $\ast  $ & $0     $ & $\ast  $ \\ 
    3        & $0     $ & $\ast 3$ & $0     $ & $\ast 3$ & $0     $ & $\ast 3$ \\  
    \hline \hline
    $2k + 4$ & $0     $ & $\ast  $ & $0     $ & $\ast  $ & $0     $ & $\ast  $ \\ 
    $2k + 5$ & $0     $ & $\ast 3$ & $0     $ & $\ast 3$ & $0     $ & $\ast 3$ \\  
   \hline
  \end{tabular}
  \end{center}
  \caption{Nimber values for $\protect \Routes{0}{0}{i}{j}{1}$.}
  \label{table:Routes2461}
\end{table}

\begin{table}[h!]
  \begin{center}
  \begin{tabular}{|r||rrrr||cc|}
  \hline
    $i$\textbackslash$j$  & $0$ & $1$ & $2$ & $3$ & $2l + 4$ & $2l + 5$ \\ \hline \hline
    0        & $0     $ & $\ast 2$ & $0     $ & $\ast 2$ & $0     $ & $\ast 2$ \\  
    1        & $0     $ & $\ast  $ & $0     $ & $\ast  $ & $0     $ & $\ast  $ \\  
    2        & $0     $ & $\ast  $ & $0     $ & $\ast  $ & $0     $ & $\ast  $ \\  
    3        & $0     $ & $\ast  $ & $0     $ & $\ast  $ & $0     $ & $\ast  $ \\  
    \hline \hline
    $2k + 4$ & $0     $ & $\ast  $ & $0     $ & $\ast  $ & $0     $ & $\ast  $ \\  
    $2k + 5$ & $0     $ & $\ast  $ & $0     $ & $\ast  $ & $0     $ & $\ast  $ \\  
   \hline
  \end{tabular}
  \end{center}
  \caption{Nimber values for $\protect \Routes{0}{0}{i}{j}{2}$.}
  \label{table:Routes2462}
\end{table}

\begin{table}[h!]
  \begin{center}
  \begin{tabular}{|r||rrrr||cc|}
  \hline
    $i$\textbackslash$j$  & $0$ & $1$ & $2$ & $3$ & $2l + 4$ & $2l + 5$ \\ \hline \hline
    0        & $0     $ & $\ast  $ & $0     $ & $\ast  $ & $0     $ & $\ast  $ \\  
    1        & $0     $ & $\ast  $ & $0     $ & $\ast  $ & $0     $ & $\ast  $ \\  
    2        & $0     $ & $\ast  $ & $0     $ & $\ast  $ & $0     $ & $\ast  $ \\  
    3        & $0     $ & $\ast  $ & $0     $ & $\ast  $ & $0     $ & $\ast  $ \\  
    \hline \hline
    $2k + 4$ & $0     $ & $\ast  $ & $0     $ & $\ast  $ & $0     $ & $\ast  $ \\  
    $2k + 5$ & $0     $ & $\ast  $ & $0     $ & $\ast  $ & $0     $ & $\ast  $ \\  
   \hline
  \end{tabular}
  \end{center}
  \caption{Nimber values for $\protect \Routes{0}{0}{i}{j}{3+}$.}
  \label{table:Routes2463}
\end{table}

\begin{table}[h!]
  \begin{center}
  \begin{tabular}{|r||rrrr||cc|}
  \hline
    $i$\textbackslash$j$  & $0$ & $1$ & $2$ & $3$ & $2l + 4$ & $2l + 5$ \\ \hline \hline
    0        & $0     $ & $\ast 2$ & $\ast  $ & $\ast  $ & $\ast  $ & $\ast  $ \\ 
    1        & $0     $ & $0     $ & $0     $ & $0     $ & $0     $ & $0    $ \\ 
    2        & $\ast  $ & $\ast 3$ & $\ast  $ & $\ast 2$ & $\ast  $ & $\ast 2$ \\  
    3        & $\ast 2$ & $0     $ & $0     $ & $0     $ & $0     $ & $0    $ \\ 
    \hline \hline
    $2k + 4$ & $\ast  $ & $\ast 3$ & $\ast  $ & $\ast 2$ & $\ast  $ & $\ast 2$ \\ 
    $2k + 5$ & $\ast 2$ & $0     $ & $0     $ & $0     $ & $0     $ & $0     $ \\ \hline
  \end{tabular}
  \end{center}
  \caption{Nimber values for $\protect \Routes{1}{0}{i}{j}{0}$.}
  \label{table:Routes35}
\end{table}

\begin{Lemma}[Route Nimbers]
\label{lemma:routeNimbers}
The nimber values given in Tables \ref{table:Routes02461} through \ref{table:Routes35} are correct.
\end{Lemma}

The proof of this lemma is an uninteresting inductive proof with a myriad of cases.  We omit most of this proof, but include some parts in the appendix.

To satisfy the attentive reader, we include a separate (sub)lemma that classifies the \outP-positions:

\begin{Lemma}\label{lemma:disjoint_paths}

The zero values are as shown in the preceding tables.

\end{Lemma}

\begin{proof}

We proceed by induction on the order of the graph. Note that the base cases are easily verified. Assume now that all such graphs are in $\Pclass$ for graphs with up to $n$ nodes, and presume that the graph in $G$ has order $n+1$. We examine each case by listing the options of each and demonstrating that all options are $\Nclass$. For the reader's benefit we will denote each option $H$ by the length of the path played on to achieve that option, or by $u$ or $v$ to denote that one of the end nodes is played. Note that playing $u$ leads to the same option as does playing $v$ unless $u$ is adjacent to pendant nodes. 

\begin{itemize}
\item[a] 
	\begin{itemize}
		\item[$u$] In this options, $u\sim v$, there are an even number of $3$-paths, and all other paths have length $5$. $u$ can be played again to a graph of the form $T_{2k,d}$ for some $k\geq 0,d>0$, which is in $\Pclass$. Thus $H\in \Nclass$.
		\item[$2$] $H$ is a graph with a non-zero collection of $6$-cycles and an even number of $4$-cycles joined at $x$. Again, playing $u$ results in $T_{2k,d}, k\geq 0, d>0$. Hence, $H\in \Nclass$.
		\item[$4$] Play in $u$ results in a graph with an odd number of $3$-paths and at least one $5$-path, which is in $\Pclass$ by inductive hypothesis.
		\item[$6$] $H = P_{2,4,6,\ldots ,6}$. Playing on $u$ results in a graph with an odd number of $3$-paths where $u\sim v$, which is in $\Pclass$ by inductive hypothesis. So $H\in \Nclass$.
	\end{itemize}

\item[b]
  If a pendant node is played, then the other pendant node is played in response.
	\begin{itemize}
		\item[$u$] If $u$ was adjacent o $v$, then $H=T_{2k+1,d},d>0$ which is in $\Nclass$ by Thm.~\ref{thm:diamonds}.  If they weren't adjacent, then $H$ has an odd number of $2$-paths and at least one $4$-path.  A play on one of the $2$-paths results in $T_{2k,d},k\geq 0,d>0$, which is in $\Pclass$. Hence, $H\in \Pclass$.
    \item[$v$] If $u$ is adjacent to a pair of pendant nodes, then $H$ is the same as above but with the pendants remaining: $T_{2k+ 3, d}, d > 0$. A play on one of the $2$-paths also results in $T_{2k+2,d},k\geq 0,d>0$, which is in $\Pclass$.
		\item[$3$] There are an even number of $3$-paths and $u\sim v$ in $H$. If $H$ also contains a $5$-path, then a play on $u$ results in $T_{2k,d},k \geq 0, d>0$, which is in $\Pclass$. Otherwise, we either have a single edge (or a star on four nodes) which is in $\Nclass$ or a positive even number of $3$-paths that can be reduced to an odd number, leaving a graph in $\Pclass$ by inductive hypothesis.
		\item[$5$] $H$ has an even number of $3$-paths. Play on one of them results in a position in $\Pclass$ by inductive hypothesis.
	\end{itemize}
	
\item[c] 
	\begin{itemize}
		\item[$u$] $H$ consists solely of at least two $3$-paths. Play on $u$ results in a $K_{2,k}$ for some $k\geq 2$. This is in $\Pclass$ by Thm.~\ref{thm:2paths}.
		\item[$4$] $H = P_{2,4,\ldots ,4}$. Play on the $2$-path results in $T_{0,d},d>0$, which is in $\Pclass$. Hence $H\in \Nclass$.
	\end{itemize}
	
\item[d] 
  If a pendant node is played, then the other pendant node is played in response.
	\begin{itemize}
		\item[$u$] $H$ contains two $3$-paths and at least two $5$-paths. Play on a $3$-path results in a game with an odd number of $3$-paths and with the end points adjacent, 
    which is in $\Pclass$ by inductive hypothesis.
    \item[$v$] If $u$ is adjacent to a pair of pendant nodes, then $H$ is as above with the pendants remaining. A play on a $3$-path is in $\Pclass$ by the same argument.
		\item[$4$] $H$ has a $2$-path, a $4$-path, and at least two $6$-paths. Playing $u$ results in an odd number of $3$-paths and $u \sim v$.  Again, in $\Pclass$ by inductive hypothesis, leading to $H\in \Nclass$. 
		\item[$6$] $H$ has three $4$-paths and at least one $6$-path. A play on $u$, yet again, results in a graph with an odd number of $3$-paths and a $5$-path. So $H\in \Nclass$.
	\end{itemize}
	
\item[e] If $G$ has a single path then this is in $\Pclass$ by Thm.~\ref{thm:paths}. Assume otherwise.
	\begin{itemize}
		\item[$u$] $H$ is a collection of $5$-paths. Play on $u$ results in a graph with all $4$-paths, which is in $\Pclass$ by inductive hypothesis.
		\item[$6$] $H$ has a $4$-path with at least one $6$ path. Play on $u$ results in a $3$-path with at least one $5$-path, which is in $\Pclass$ by inductive hypothesis. Hence $H\in \Nclass$.
	\end{itemize}

\end{itemize}

Since all options of all relevant positions are in $\Nclass$, these positions are in $\Pclass$.
\end{proof}

It is worth noting that a computer search of connected bipartite graphs has yielded no games with grundy value larger than $13$. It remains an open problem whether or not \fc\ has a finite nim dimension.

\section{Real flag values}\label{sec:flags}

Below we examine some real-world flags and their values when played as games of \fc.

The current flag of the United States of America (Fig.~\ref{fig:us}) is easily modeled by a graph with a single blue vertex with fifty white pendant neighbors, and adjacent to eight vertices on a line of thirteen alternating red and white vertices. This game is in $\Pclass$, with an argument as follows:
\begin{itemize}
\item If a white star vertex is turned blue, then another star node can be countered in the same way.
\item If the blue node is turned white, then what remains is a white node with four red pendant neighbors, at the end of an alternating path of length five. Playing $w_6$ results in a path of length three with an even number of pendants, which is in $\Pclass$ by Theorems \ref{thm:paths} and \ref{thm:pendantsPath}. 
\item If the blue node is turned red, then a similar situation arises. What remains is a red node with an odd number of white pendant neighbors (the star nodes and $w_1, w_2, w_3$, at the end of an alternating path of length six. This is in $\Nclass$.
\item Any play on a node along the path can be countered with either another path node, retaining relative parity of the colors, or by playing the blue node to create a path surmounted with appropriate number of pendant nodes.
\end{itemize}

\begin{figure}[H]
\centering
\raisebox{0.4\height}{\fboxsep=0pt \fbox{\includegraphics[width=.45\textwidth]{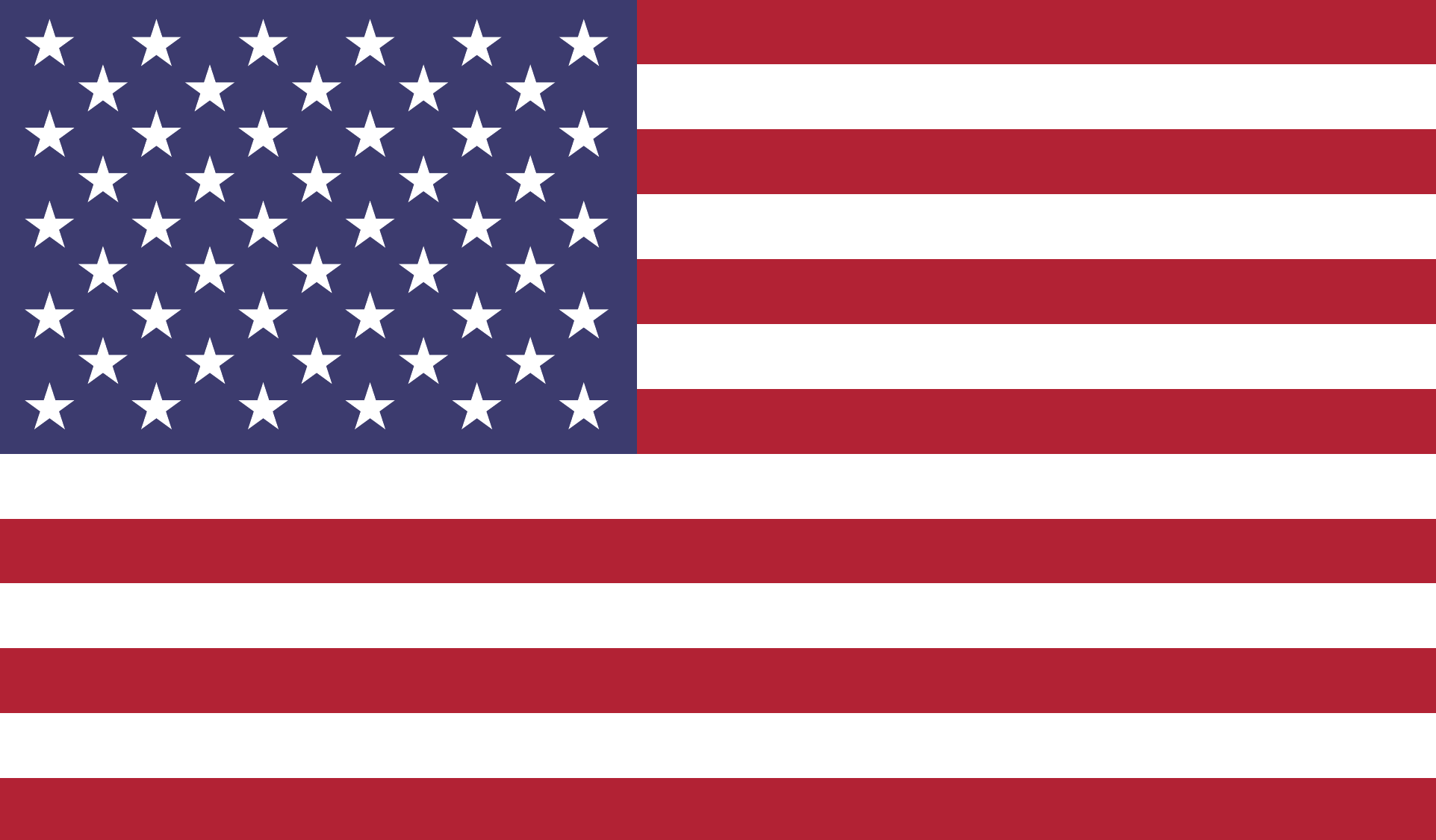}}}
\definecolor{eqeqeq}{rgb}{0.8784313725490196,0.8784313725490196,0.8784313725490196}
\begin{tikzpicture}[line cap=round,line join=round,>=triangle 45,x=1.0cm,y=1.0cm,scale=0.4]
\clip(21.951015325233268,-8.558978418637924) rectangle (36.43726757281635,7.67964305244311);
\draw [line width=1.pt] (30.,0.)-- (30.,5.);
\draw [line width=1.pt] (30.,0.)-- (28.142719914926044,4.700615372496589);
\draw [line width=1.pt] (30.,0.)-- (25.10527992752959,0.);
\draw [line width=1.pt] (32.99094143327038,4.992676909746248)-- (32.99094143327038,-1.0249248702679399);
\draw [line width=1.pt] (32.99094143327038,-1.0249248702679399)-- (32.99094143327038,-6.865021502589905);
\draw [line width=1.pt] (32.99094143327038,4.992676909746248)-- (30.,0.);
\draw [line width=1.pt] (32.99094143327038,-1.0249248702679399)-- (30.,0.);
\begin{scriptsize}
\draw [fill=blue] (30.,0.) circle (7.5pt);
\draw[color=black] (29.369378371374605,-1.0707692156207334) node {$b$};
\draw [fill=white] (30.,5.) circle (7.5pt);
\draw[color=black] (29.917304518797796,6.219335366194817) node {$s_1$};
\draw [fill=white] (28.142719914926044,4.700615372496589) circle (7.5pt);
\draw[color=black] (26.945267612202436,4.9342646022963175) node {$s_2$};
\draw [fill=white] (25.10527992752959,0.) circle (7.5pt);
\draw[color=black] (24.08306454715578,0.9622276957009572) node {$s_{50}$};
\draw [fill=black] (27.024947610626995,2.981895378798362) circle (3.0pt);
\draw [fill=black] (26.623999921227817,2.422535381949249) circle (3.0pt);
\draw [fill=black] (26.39035069142809,1.8127015404005455) circle (3.0pt);
\draw [fill=red] (32.99094143327038,4.992676909746248) circle (7.5pt);
\draw[color=black] (33.89633219874432,6.219335366194817) node {$r_1$};
\draw [fill=white] (32.99094143327038,-1.0249248702679399) circle (7.5pt);
\draw[color=black] (33.95474450619425,-0.7901415277969959) node {$w_4$};
\draw [fill=red] (32.99094143327038,-6.865021502589905) circle (7.5pt);
\draw[color=black] (34.24680604344391,-6.397723042990445) node {$r_7$};
\draw [fill=black] (32.99094143327038,3.1234830713484305) circle (3.0pt);
\draw [fill=black] (32.99094143327038,2.5720615372496587) circle (3.0pt);
\draw [fill=black] (32.99094143327038,1.979052310600819) circle (3.0pt);
\draw [fill=black] (32.99094143327038,-2.754337979392339) circle (3.0pt);
\draw [fill=black] (32.99094143327038,-3.36225747392267) circle (3.0pt);
\draw [fill=black] (32.99094143327038,-4.037415356742159) circle (3.0pt);
\end{scriptsize}
\end{tikzpicture}
\caption{The US flag graph, which is in $\Pclass$}
\label{fig:us}
\end{figure}

Let's look at the Flag of the Portuguese Autonomous Region of the Azores (Fig.~\ref{fig:azores}). If we consider a simplified version wherein the goshawk is a single yellow region, then we get the graph in the figure. An exhaustive search of options (not included) demonstrates that this graph is in $\Nclass$, with a winning first move being the coloring of the red shield blue. From there, given the parity of many of the flag's elements, countermoves are possible for every remaining option. 

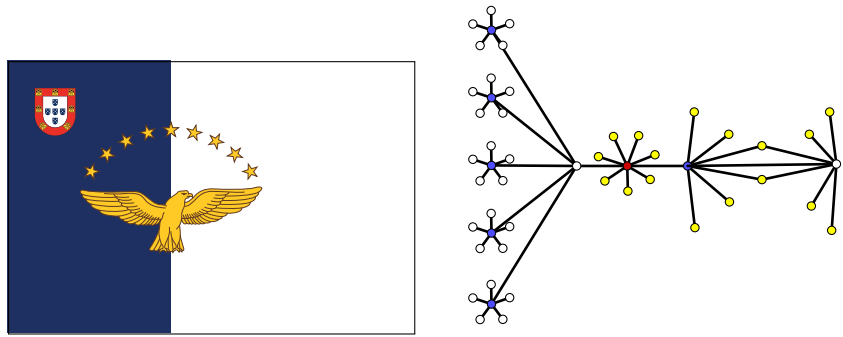
\begin{figure}[H]
\centering
\fboxsep=-1pt
\fbox{\includegraphics[width=.45\textwidth]{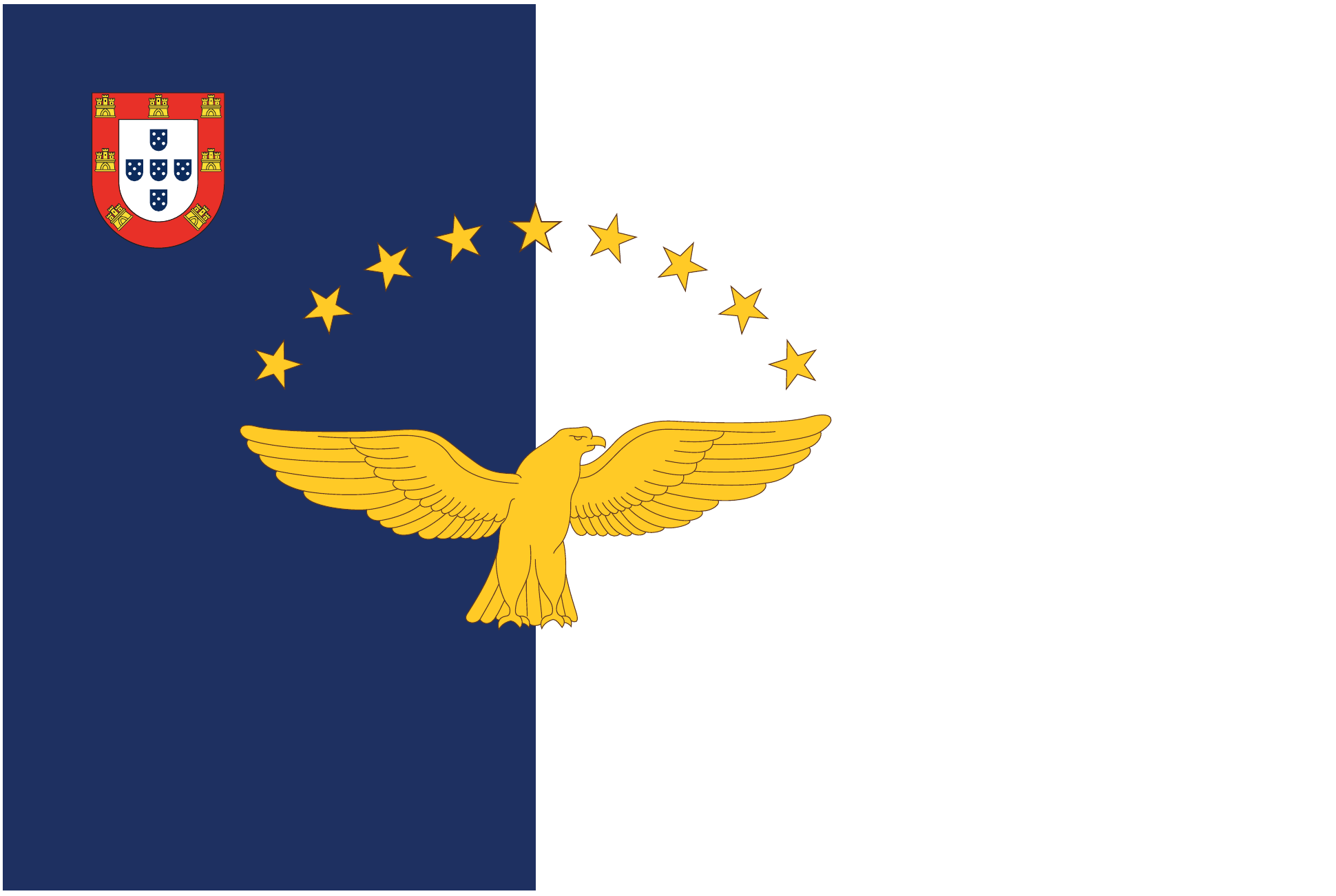}}
\definecolor{ffffqq}{rgb}{1.,1.,0.}
\definecolor{ccqqqq}{rgb}{0.8,0.,0.}
\definecolor{ududff}{rgb}{0.30196078431372547,0.30196078431372547,1.}
\begin{tikzpicture}[line cap=round,line join=round,>=triangle 45,x=1.0cm,y=1.0cm,scale=0.45]
\clip(-3.2,-4.58) rectangle (8.66,5.52);
\draw [line width=1.pt] (-1.99,4.42)-- (-2.,5.);
\draw [line width=1.pt] (-1.99,4.42)-- (-1.4527082165194667,4.610036278950596);
\draw [line width=1.pt] (-1.99,4.42)-- (-1.654463292564949,3.969025444962681);
\draw [line width=1.pt] (-1.99,4.42)-- (-2.32644657044441,3.9628226834506375);
\draw [line width=1.pt] (-1.99,4.42)-- (-2.54,4.6);
\draw [line width=1.pt] (-1.99,2.42)-- (-2.,3.);
\draw [line width=1.pt] (-1.99,2.42)-- (-1.452708216519467,2.6100362789505955);
\draw [line width=1.pt] (-1.99,2.42)-- (-1.6544632925649496,1.969025444962681);
\draw [line width=1.pt] (-1.99,2.42)-- (-2.3264465704444106,1.9628226834506377);
\draw [line width=1.pt] (-1.99,2.42)-- (-2.54,2.6);
\draw [line width=1.pt] (-1.99,0.42)-- (-2.,1.);
\draw [line width=1.pt] (-1.99,0.42)-- (-1.452708216519467,0.610036278950596);
\draw [line width=1.pt] (-1.99,0.42)-- (-1.6544632925649494,-0.030974555037318385);
\draw [line width=1.pt] (-1.99,0.42)-- (-2.32644657044441,-0.03717731654936174);
\draw [line width=1.pt] (-1.99,0.42)-- (-2.54,0.6);
\draw [line width=1.pt] (-1.99,-1.58)-- (-2.,-1.);
\draw [line width=1.pt] (-1.99,-1.58)-- (-1.4527082165194671,-1.3899637210494042);
\draw [line width=1.pt] (-1.99,-1.58)-- (-1.6544632925649498,-2.0309745550373184);
\draw [line width=1.pt] (-1.99,-1.58)-- (-2.3264465704444106,-2.0371773165493616);
\draw [line width=1.pt] (-1.99,-1.58)-- (-2.54,-1.4);
\draw [line width=1.pt] (-1.99,-3.68)-- (-2.,-3.1);
\draw [line width=1.pt] (-1.99,-3.68)-- (-1.4527082165194671,-3.489963721049404);
\draw [line width=1.pt] (-1.99,-3.68)-- (-1.6544632925649498,-4.1309745550373185);
\draw [line width=1.pt] (-1.99,-3.68)-- (-2.3264465704444106,-4.137177316549361);
\draw [line width=1.pt] (-1.99,-3.68)-- (-2.54,-3.5);
\draw [line width=1.pt] (-1.99,4.42)-- (0.5263727416662906,0.40362725833370927);
\draw [line width=1.pt] (0.5263727416662906,0.40362725833370927)-- (-1.99,2.42);
\draw [line width=1.pt] (0.5263727416662906,0.40362725833370927)-- (-1.99,0.42);
\draw [line width=1.pt] (0.5263727416662906,0.40362725833370927)-- (-1.99,-1.58);
\draw [line width=1.pt] (0.5263727416662906,0.40362725833370927)-- (-1.99,-3.68);
\draw [line width=1.pt] (0.5263727416662906,0.40362725833370927)-- (2.0204428080198737,0.40362725833370927);
\draw [line width=1.pt] (2.0204428080198737,0.40362725833370927)-- (1.6143689882213115,1.2748870487326263);
\draw [line width=1.pt] (2.0204428080198737,0.40362725833370927)-- (2.357192940130224,1.2996373309618399);
\draw [line width=1.pt] (2.0204428080198737,0.40362725833370927)-- (2.8396866485686014,0.7343057279911742);
\draw [line width=1.pt] (2.0204428080198737,0.40362725833370927)-- (2.6985225100043477,0.0045984675206399706);
\draw [line width=1.pt] (2.0204428080198737,0.40362725833370927)-- (2.04,-0.34);
\draw [line width=1.pt] (2.0204428080198737,0.40362725833370927)-- (1.36,-0.04);
\draw [line width=1.pt] (2.0204428080198737,0.40362725833370927)-- (1.1705763794764703,0.6786923486358806);
\draw [line width=1.pt] (2.0204428080198737,0.40362725833370927)-- (3.8030146602036674,0.40362725833370927);
\draw [line width=1.pt] (3.8030146602036674,0.40362725833370927)-- (6.,1.);
\draw [line width=1.pt] (6.,1.)-- (8.2,0.46);
\draw [line width=1.pt] (8.2,0.46)-- (6.,0.);
\draw [line width=1.pt] (6.,0.)-- (3.8030146602036674,0.40362725833370927);
\draw [line width=1.pt] (3.8030146602036674,0.40362725833370927)-- (5.02,1.34);
\draw [line width=1.pt] (3.8030146602036674,0.40362725833370927)-- (4.,2.);
\draw [line width=1.pt] (3.8030146602036674,0.40362725833370927)-- (4.02,-1.42);
\draw [line width=1.pt] (3.8030146602036674,0.40362725833370927)-- (5.04,-0.66);
\draw [line width=1.pt] (8.2,0.46)-- (7.4,1.34);
\draw [line width=1.pt] (8.2,0.46)-- (8.,2.);
\draw [line width=1.pt] (8.2,0.46)-- (7.46,-0.78);
\draw [line width=1.pt] (8.2,0.46)-- (8.06,-1.5);
\begin{scriptsize}
\draw [color=black,fill=white] (-2.,5.) circle (3.5pt);
\draw [color=black,fill=white] (-2.54,4.6) circle (3.5pt);
\draw [color=black,fill=white] (-2.32644657044441,3.9628226834506375) circle (3.5pt);
\draw [color=black,fill=white] (-1.654463292564949,3.969025444962681) circle (3.5pt);
\draw [color=black,fill=white] (-1.4527082165194667,4.610036278950596) circle (3.5pt);
\draw [fill=ududff] (-1.99,4.42) circle (3.5pt);
\draw [color=black,fill=white] (-2.,3.) circle (3.5pt);
\draw [color=black,fill=white] (-2.54,2.6) circle (3.5pt);
\draw [color=black,fill=white] (-2.3264465704444106,1.9628226834506377) circle (3.5pt);
\draw [color=black,fill=white] (-1.6544632925649496,1.969025444962681) circle (3.5pt);
\draw [color=black,fill=white] (-1.452708216519467,2.6100362789505955) circle (3.5pt);
\draw [fill=ududff] (-1.99,2.42) circle (3.5pt);
\draw [color=black,fill=white] (-2.,1.) circle (3.5pt);
\draw [color=black,fill=white] (-2.54,0.6) circle (3.5pt);
\draw [color=black,fill=white] (-2.32644657044441,-0.03717731654936174) circle (3.5pt);
\draw [color=black,fill=white] (-1.6544632925649494,-0.030974555037318385) circle (3.5pt);
\draw [color=black,fill=white] (-1.452708216519467,0.610036278950596) circle (3.5pt);
\draw [fill=ududff] (-1.99,0.42) circle (3.5pt);
\draw [color=black,fill=white] (-2.,-1.) circle (3.5pt);
\draw [color=black,fill=white] (-2.54,-1.4) circle (3.5pt);
\draw [color=black,fill=white] (-2.3264465704444106,-2.0371773165493616) circle (3.5pt);
\draw [color=black,fill=white] (-1.6544632925649498,-2.0309745550373184) circle (3.5pt);
\draw [color=black,fill=white] (-1.4527082165194671,-1.3899637210494042) circle (3.5pt);
\draw [fill=ududff] (-1.99,-1.58) circle (3.5pt);
\draw [color=black,fill=white] (-2.,-3.1) circle (3.5pt);
\draw [color=black,fill=white] (-2.54,-3.5) circle (3.5pt);
\draw [color=black,fill=white] (-2.3264465704444106,-4.137177316549361) circle (3.5pt);
\draw [color=black,fill=white] (-1.6544632925649498,-4.1309745550373185) circle (3.5pt);
\draw [color=black,fill=white] (-1.4527082165194671,-3.489963721049404) circle (3.5pt);
\draw [fill=ududff] (-1.99,-3.68) circle (3.5pt);
\draw [color=black,fill=white] (0.5263727416662906,0.40362725833370927) circle (3.5pt);
\draw [fill=ccqqqq] (2.0204428080198737,0.40362725833370927) circle (3.5pt);
\draw [fill=ffffqq] (1.36,-0.04) circle (3.5pt);
\draw [fill=ffffqq] (2.04,-0.34) circle (3.5pt);
\draw [fill=ffffqq] (2.6985225100043477,0.0045984675206399706) circle (3.5pt);
\draw [fill=ffffqq] (2.8396866485686014,0.7343057279911742) circle (3.5pt);
\draw [fill=ffffqq] (2.357192940130224,1.2996373309618399) circle (3.5pt);
\draw [fill=ffffqq] (1.6143689882213115,1.2748870487326263) circle (3.5pt);
\draw [fill=ffffqq] (1.1705763794764703,0.6786923486358806) circle (3.5pt);
\draw [fill=ududff] (3.8030146602036674,0.40362725833370927) circle (3.5pt);
\draw [line width=1.pt] (3.8030146602036674,0.40362725833370927)-- (8.2,0.46);
\draw [fill=ffffqq] (6.,1.) circle (3.5pt);
\draw [color=black,fill=white] (8.2,0.46) circle (3.5pt);
\draw [fill=ffffqq] (6.,0.) circle (3.5pt);
\draw [fill=ffffqq] (5.02,1.34) circle (3.5pt);
\draw [fill=ffffqq] (4.,2.) circle (3.5pt);
\draw [fill=ffffqq] (4.02,-1.42) circle (3.5pt);
\draw [fill=ffffqq] (5.04,-0.66) circle (3.5pt);
\draw [fill=ffffqq] (7.4,1.34) circle (3.5pt);
\draw [fill=ffffqq] (8.,2.) circle (3.5pt);
\draw [fill=ffffqq] (7.46,-0.78) circle (3.5pt);
\draw [fill=ffffqq] (8.06,-1.5) circle (3.5pt);
\end{scriptsize}
\end{tikzpicture}
\caption{The graph associated with the flag of the Azores, (simplified)}
\label{fig:azores}
\end{figure}

Classes and Grundy values of flags with much simpler designs are easier to determine. For example, it's clear from inspection that the flags of Canada, France, and the modernized historical Maine flag (Fig.~\ref{fig:other_flags}) have the values $*,0,0$, respectively.

\begin{figure}[H]
\centering
\fboxsep=0pt
\fbox{\includegraphics[width=0.4\textwidth]{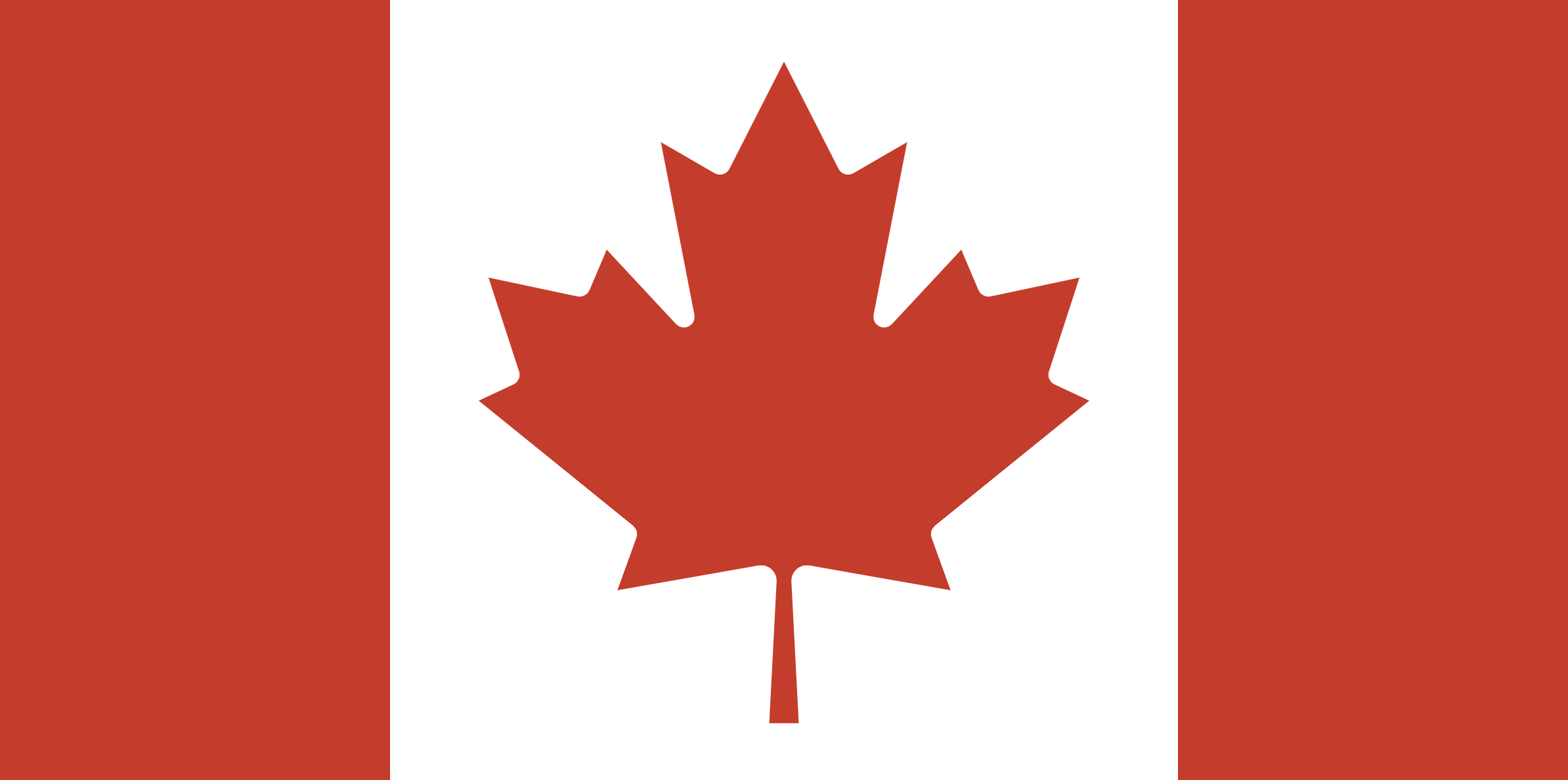}}
\hspace{.5cm}
\fboxsep=0pt
\fbox{\includegraphics[width=0.3\textwidth]{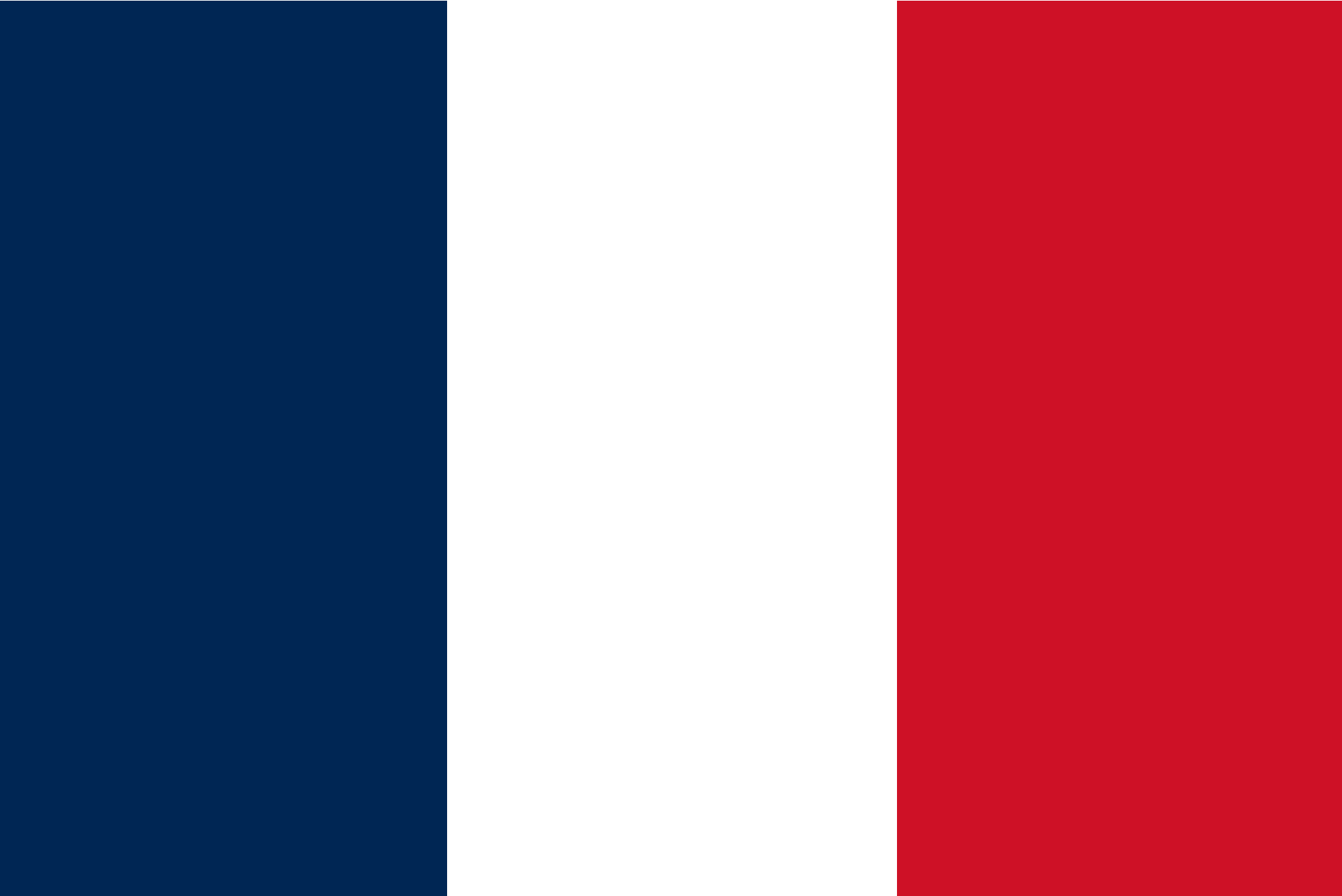}}
\vspace{.5cm}

\fboxsep=0pt
\fbox{\includegraphics[width=0.3\textwidth]{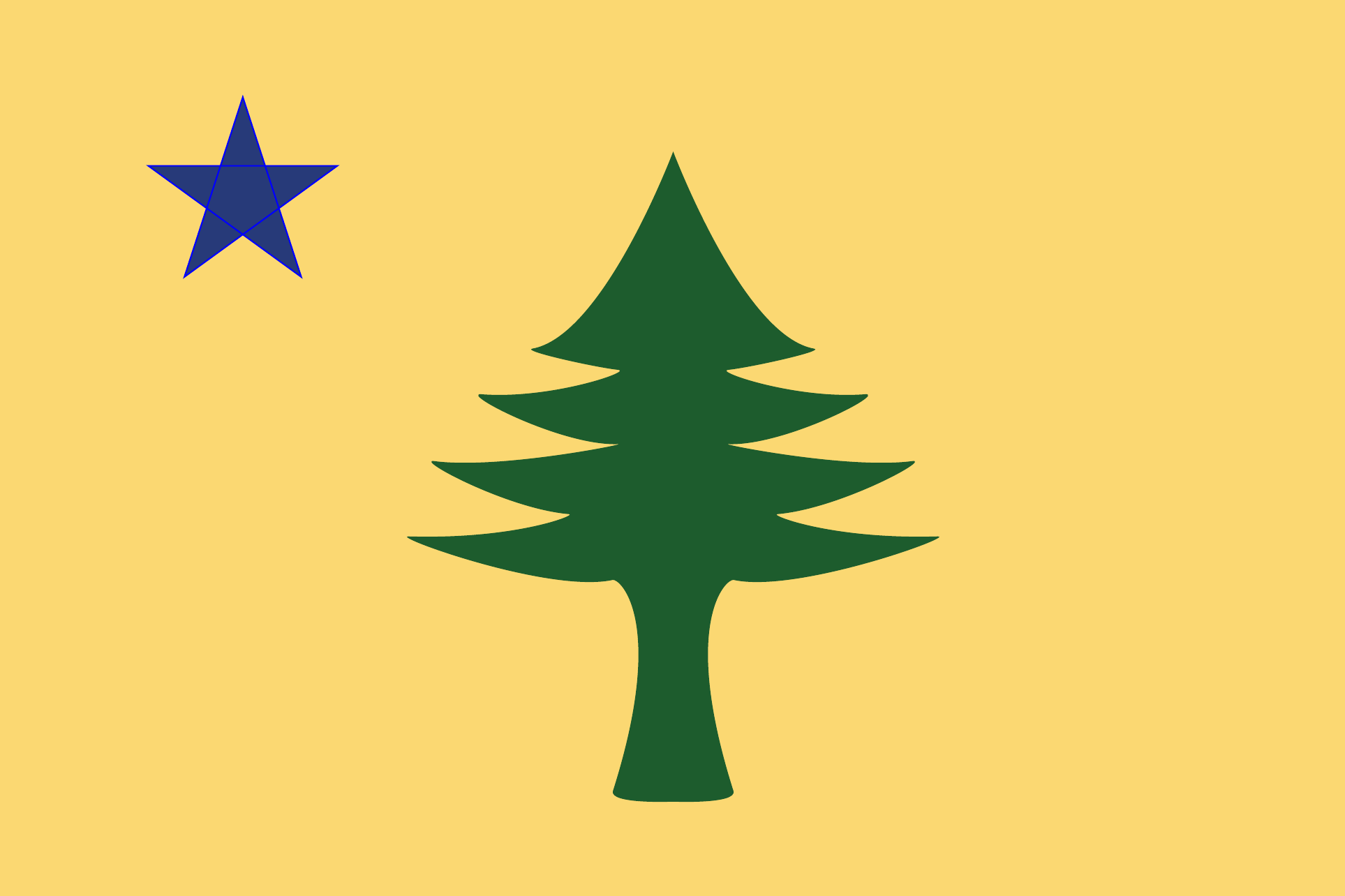}}
\caption{Flag of Canada, Flag of France, Modernized historical flag of Maine}
\label{fig:other_flags}
\end{figure}

\section{Computational Complexity}\label{sec:complexity}

We now turn to the computational complexity of \fc. First, we remind the reader of another game, \at, known to be \cclass{PSPACE}-complete \cite{DBLP:journals/jcss/Schaefer78}.

\begin{Ruleset}
\at~is an impartial game played on a Positive (no negations) CNF formula of boolean variables.  Initially, all variables are assigned the value False.  A turn consists of picking a single variable still assigned False, and switching the value to True, with the condition that the formula continues to evaluate to False.  A variable change that would cause the entire formula to become True cannot be selected.
\end{Ruleset}

\begin{Theorem}\label{thm:complexity}
The game \fc~is \cclass{PSPACE}-hard.
\end{Theorem}
\begin{proof}
Our goal is to reduce any game of \at~to a game of \fc, demonstrating that \fc~is at least as hard as \at. Since \at\ is \cclass{PSPACE}-hard \cite{DBLP:journals/jcss/Schaefer78}, \fc\ is also \cclass{PSPACE}-hard.  We continue by showing the reduction from \at\ to \fc.

Let $A$ be an instance of \at~with boolean variables $\{x_1,\ldots ,x_k\}$ and clauses $\{C_1,\ldots ,C_n\}$. Define the two-colored graph $G(A)$ in the following way.

Let $x$ and $x'$ be black nodes, and let $X=\{x_1,\ldots ,x_k\}$ be white \emph{variable nodes} representing the variables in $A$ adjacent to both $x$ and $x'$. Next, let $C_1,\ldots ,C_n$ be black nodes representing the clauses from $A$, adjacent to $X$ as determined by $A$. Clone each three times, so there are four copies of $C_i$ for each $i$ - $\{C_i,C_i^2,C_i^3,C_i^4\}$. We also include black nodes $C_\gamma, C_\delta, C_\epsilon$, all of which are adjacent to all of $X$. Thus, we have $4n+3$ \emph{clause nodes} in $C$. For each clause node $C_j\in C$ we create the path $C_ja_jb_jc_jd_j$, properly two-colored as necessary. Finally, adjoin every $d_j$ to a white vertex $y$. See Fig.~\ref{fig:reduction}.

\begin{figure}[H]
\centering
\begin{tikzpicture}[line cap=round,line join=round,>=triangle 45,x=1.0cm,y=1.0cm,scale=1.25]
\clip(-3.62,1.46) rectangle (4.74,6.88);
\draw [line width=1.pt] (-3.,4.5)-- (-2.,6.);
\draw [line width=1.pt] (-3.,4.5)-- (-2.,5.);
\draw [line width=1.pt] (-3.,4.5)-- (-2.,2.);
\draw [line width=1.pt] (-3.,3.5)-- (-2.,6.);
\draw [line width=1.pt] (-3.,3.5)-- (-2.,5.);
\draw [line width=1.pt] (-3.,3.5)-- (-2.,2.);
\draw [line width=1.pt] (-1.,6.)-- (0.,6.);
\draw [line width=1.pt] (0.,6.)-- (1.,6.);
\draw [line width=1.pt] (1.,6.)-- (2.,6.);
\draw [line width=1.pt] (2.,6.)-- (3.,6.);
\draw [line width=1.pt] (-1.,5.)-- (0.,5.);
\draw [line width=1.pt] (0.,5.)-- (1.,5.);
\draw [line width=1.pt] (1.,5.)-- (2.,5.);
\draw [line width=1.pt] (2.,5.)-- (3.,5.);
\draw [line width=1.pt] (-1.,2.)-- (0.,2.);
\draw [line width=1.pt] (0.,2.)-- (1.,2.);
\draw [line width=1.pt] (1.,2.)-- (2.,2.);
\draw [line width=1.pt] (2.,2.)-- (3.,2.);
\draw [line width=1.pt] (3.,2.)-- (4.,4.);
\draw [line width=1.pt] (4.,4.)-- (3.,5.);
\draw [line width=1.pt] (3.,6.)-- (4.,4.);
\draw [fill=black] (-3.,4.5) circle (2.5pt);
\draw[color=black] (-3.32,4.73) node {$x$};
\draw [fill=black] (-3.,3.5) circle (2.5pt);
\draw[color=black] (-3.32,3.71) node {$x'$};
\draw [color=black,fill=white] (2.,6.) circle (2.5pt);
\draw[color=black] (2.03,6.26) node {$c_1$};
\draw [color=black,fill=white] (2.,5.) circle (2.5pt);
\draw[color=black] (1.99,5.32) node {$c_2$};
\draw [color=black,fill=white] (2.,2.) circle (2.5pt);
\draw[color=black] (2.45,2.26) node {$c_\epsilon$};
\draw [color=black,fill=white] (-2.,6.) circle (2.5pt);
\draw[color=black] (-2.03,6.24) node {$x_1$};
\draw [color=black,fill=white] (-2.,5.) circle (2.5pt);
\draw[color=black] (-2.05,5.3) node {$x_2$};
\draw [color=black,fill=white] (-2.,2.) circle (2.5pt);
\draw[color=black] (-1.81,2.28) node {$x_k$};
\draw [fill=black] (-1.,6.) circle (2.5pt);
\draw[color=black] (-1.05,6.38) node {$C_1$};
\draw [fill=black] (-1.,5.) circle (2.5pt);
\draw[color=black] (-1.07,5.4) node {$C_2$};
\draw [fill=black] (-1.,2.) circle (2.5pt);
\draw[color=black] (-0.63,2.4) node {$C_\epsilon$};
\draw [fill=black] (1.,6.) circle (2.5pt);
\draw[color=black] (1.03,6.34) node {$b_1$};
\draw [fill=black] (1.,5.) circle (2.5pt);
\draw[color=black] (1.03,5.38) node {$b_2$};
\draw [fill=black] (1.,2.) circle (2.5pt);
\draw[color=black] (1.45,2.36) node {$b_\epsilon$};
\draw [color=black,fill=white] (0.,6.) circle (2.5pt);
\draw[color=black] (-0.03,6.26) node {$a_1$};
\draw [color=black,fill=white] (0.,5.) circle (2.5pt);
\draw[color=black] (-0.03,5.28) node {$a_2$};
\draw [color=black,fill=white] (0.,2.) circle (2.5pt);
\draw[color=black] (0.41,2.3) node {$a_\epsilon$};
\draw [fill=black] (3.,6.) circle (2.5pt);
\draw[color=black] (2.99,6.34) node {$d_1$};
\draw [fill=black] (3.,5.) circle (2.5pt);
\draw[color=black] (2.99,5.38) node {$d_2$};
\draw [fill=black] (3.,2.) circle (2.5pt);
\draw[color=black] (3.41,2.36) node {$d_\epsilon$};
\draw [color=black,fill=white] (4.,4.) circle (2.5pt);
\draw[color=black] (4.36,4.23) node {$y$};
\end{tikzpicture}
\caption{Construction of the game $G(A)$ from the \at~game $A$}
\label{fig:reduction}
\end{figure}
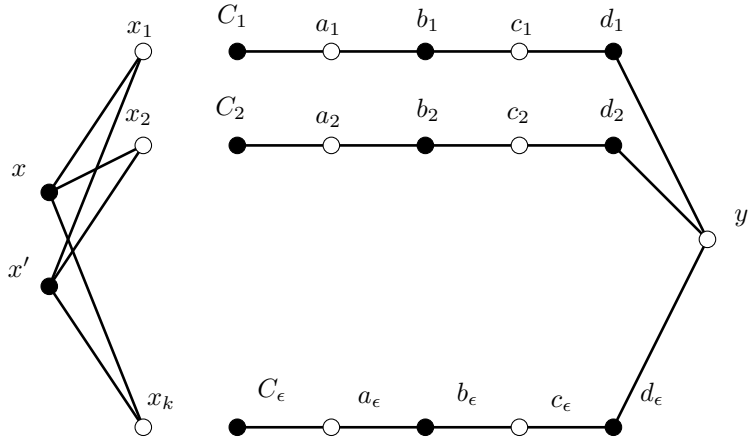

Any legal move in $A$ is associated with the coloring of some $x_j$ in $G(A)$. A game in $A$ is over when a player cannot move without selecting any remaining variable without every clause containing a chosen variable. In $G(A)$ this is associated with every clause node being \emph{saturated}, i.e. adjacent to a black variable node. We now demonstrate that every move in $G(A)$ that is not associated with a legal move in $A$ results in a \fc ~position in $\Nclass$.

If every clause node is saturated, then $G(A)$ is a black node $x$ with a pendant for every remaining white variable node. The vertices $x$ and $y$ are then joined by $2n+3$ paths of length $5$. Playing on $x$ results in a graph of the form $P_{4,\ldots ,4}$, which is in $\Pclass$ by Lemma~\ref{lemma:disjoint_paths}. 

Next we examine the case where something other than $x_i$ is played while all of $X$ is still white. If $C_i$ is played, then a move on $C_\gamma$ (or $C_\delta$ if $i=\gamma$) results in $P_{4,4,6,\ldots ,6}$ with a pair of pendants for the vertices $x$ and $x'$. This is in $\Pclass$ by Lemma~\ref{lemma:disjoint_paths}. Similarly, if any vertex in $\{a_i,b_i,c_i,d_i\}$ is played then $C_\gamma$ (or $C_\delta$ if $i=\gamma$) results in the same graph in $\Pclass$.If $x$ is played then playing $x'$ results in $P_{6,\ldots ,6}$ which is in $\Pclass$ by Lemma~\ref{lemma:disjoint_paths}. Finally, if $y$ is played, then $C_\gamma$ results in $P_{3,5,\ldots 5}$ with two pendants for $x$ and $x'$, which is in $\Pclass$ by Lemma~\ref{lemma:disjoint_paths}.

Now we examine the case where at least one vertex in $X$ is black. If $x$ is played (or any vertex with a black path to $x$), then playing an unsaturated vertex $C_i$ in $C$ results in $P_{4,\ldots ,4,6,\ldots ,6}$, where there are an even number of $4$-paths (associated with the saturated clause nodes along with $C_i$) and an odd number of $6$-paths (associated with the unsaturated clause nodes, excluding $C_i$). There are at least two $6$-paths due to the extra copies of $C_i$, so this is in $\Pclass$ by Lemma~\ref{lemma:disjoint_paths}. The same argument demonstrates that playing an unsaturated $C_i$ or a vertex in $\{a_i,b_i,c_i,d_i\}$ first can be countered with $C_\gamma$. If a vertex in $\{a_i,b_i,c_i,d_i\}$  is played where $C_i$ is saturated, then a play on $C_\gamma$ results in $P_{2,4,\ldots ,4,6,\ldots ,6}$, where the number of $4$-paths is even, which is in $\Pclass$ by the Lemma. Finally, if $y$ is played then a play on $C_\gamma$ results in $P_{3,\ldots ,3,5,\ldots ,5}$, where there are an odd number of $3$-paths (associated with saturated clause nodes). This is in $\Pclass$ by Lemma~\ref{lemma:disjoint_paths}. 

Hence, any move in $G(A)$ associated with an illegal move in $A$ results in a position easily moved to $\Pclass$ by the next player, and therefore any game of \at~can be reduced to a game of \fc. 
\end{proof}

\section{Conclusion and open questions}\label{sec:conclusion}
As we have seen, \fc\ is simple to play yet yields interesting mathematical properties. It joins the ranks of other PSPACE-complete games, and thus provides yet another tool in the box for complexity analysis of other games and algorithms. As noted in Section \ref{sec:intro}, it remains to be determined whether or not \fc\ nimbers are bounded (i.e. if it has finite \emph{nim dimension}). As we have been as yet unable to find a simpler reduction, we also suspect that \fc\ might be in P if restricted to trees and possibly even planar graphs. Similarly, the question of values of grid graphs would make for an interesting addition to the results. We list these open problems below for the ambitious reader.

\begin{Open}
Is the nim dimension of $2$-color \fc\ finite?
\end{Open}

\begin{Open}
Is the nim dimension of $n$-color \fc\ finite?
\end{Open}

\begin{Open}
Is \fc\ in P when restricted to planar graphs or trees?
\end{Open}

\begin{Open}
What are the values of $n$-color \fc\ games on grid graphs?
\end{Open}

\bibliographystyle{plainurl} 

\appendix
\section{Long Proofs}

Parts of proof of Lemma \ref{lemma:routeNimbers}.
\begin{proof}

We omit the base cases for the tables, which are needed to show the top rows and left-hand columns.  These are not trivial, as they require inductive arguments of their own.  For reasons of brevity, we show only the inductive cases.

We begin by writing the four main categories of positions as the impartial position consisting of their options.  Each position has one of three or four possible moves: change the color of either $u$ or $v$ (this shortens the lengths of all paths by 1, unless $u = v$ in which case all paths are shortened by 2), or change the color of a (non-$u$-or-$v$) vertex along one of the paths (this shortens that path length by 2).  Thus, we have the following characterizations, listed in the order that we prove each of them:

\begin{itemize}
    \item $\Routes{0}{1+}{i}{j}{k} = \impartialGameSet{\Routes{0}{1+}{j}{k}{0}, \Routes{0}{1+}{i-1}{j}{k}, \Routes{0}{1+}{i+1}{j-1}{k}, \Routes{0}{1+}{i}{j+1}{k-1}}$
    
    \item $\Routes{1}{1+}{i}{j}{0} = \impartialGameSet{\Routes{0}{1+}{i}{j}{0}, \Routes{1}{1+}{i-1}{j}{0}, \Routes{1}{1+}{i+1}{j-1}{0}}$
    
    \item $\Routes{0}{0}{i}{j}{k} = \impartialGameSet{\Routes{1}{1+}{j}{k}{0}, \Routes{0}{1+}{i-1}{j}{k}, \Routes{0}{0}{i+1}{j-1}{k}, \Routes{0}{0}{i}{j+1}{k-1}}$
    
    \item $\Routes{1}{0}{i}{j}{0} = \impartialGameSet{\Routes{0}{0}{i}{j}{0}, \Routes{1}{1+}{i-1}{j}{0}, \Routes{1}{0}{i+1}{j-1}{0}}$
\end{itemize}

For each of the tables, there is a repeating $2 \times 2$ block as $i$ and $j$ grow.  To prove each of those repeats, we recursively calculate them based on the mex of their options.

Beginning with the $\Routes{0}{1+}{i}{j}{k}$ case, we prove this for increasing the number of length-six paths as the tables do.  (We don't need to cover the zero case, as $\Routes{0}{1+}{i}{j}{0} = T_{i,j}$.)  For greater values, we prove the even and odd cases for both $i$ and $j$ in Table \ref{table:Routes02461}:
\begin{align*}
    \Routes{0}{1+}{2k+4}{2l+4}{1} 
        &= \impartialGameSet{\Routes{0}{1+}{2l+4}{1}{0}, \Routes{0}{1+}{2k+3}{2l+4}{1}, \Routes{0}{1+}{2k+5}{2l+3}{1}, \Routes{0}{1+}{2k+4}{2l+5}{0}}\\
        &= \impartialGameSet{0, \ast 3, 0, 0} \\
        &= \ast\ \\
    \Routes{0}{1+}{2k+4}{2l+5}{1} 
        &= \impartialGameSet{\Routes{0}{1+}{2l+5}{1}{0}, \Routes{0}{1+}{2k+3}{2l+5}{1}, \Routes{0}{1+}{2k+5}{2l+4}{1}, \Routes{0}{1+}{2k+4}{2l+6}{0}}\\
        &= \impartialGameSet{\ast 3, 0, \ast 3, 0} \\
        &= \ast\ \\
    \Routes{0}{1+}{2k+5}{2l+4}{1} 
        &= \impartialGameSet{\Routes{0}{1+}{2l+4}{1}{0}, \Routes{0}{1+}{2k+4}{2l+4}{1}, \Routes{0}{1+}{2k+6}{2l+3}{1}, \Routes{0}{1+}{2k+5}{2l+5}{0}}\\
        &= \impartialGameSet{0, \ast, \ast, \ast 2} \\
        &= \ast 3\ \\
    \Routes{0}{1+}{2k+5}{2l+5}{1} 
        &= \impartialGameSet{\Routes{0}{1+}{2l+5}{1}{0}, \Routes{0}{1+}{2k+4}{2l+5}{1}, \Routes{0}{1+}{2k+6}{2l+4}{1}, \Routes{0}{1+}{2k+5}{2l+6}{0}}\\
        &= \impartialGameSet{\ast 3, \ast, \ast, \ast } \\
        &= 0\ \\
\end{align*}

Next, for 2 length-six paths in Table \ref{table:Routes02462} $\Routes{0}{1+}{i}{j}{2}$:
\begin{align*}
    \Routes{0}{1+}{2k+4}{2l+4}{2} 
        &= \impartialGameSet{\Routes{0}{1+}{2l+4}{2}{0}, \Routes{0}{1+}{2k+3}{2l+4}{2}, \Routes{0}{1+}{2k+5}{2l+3}{2}, \Routes{0}{1+}{2k+4}{2l+5}{1}}\\
        &= \impartialGameSet{0, \ast, \ast 4, \ast} \\
        &= \ast 2\ \\
    \Routes{0}{1+}{2k+4}{2l+5}{2} 
        &= \impartialGameSet{\Routes{0}{1+}{2l+5}{2}{0}, \Routes{0}{1+}{2k+3}{2l+5}{2}, \Routes{0}{1+}{2k+5}{2l+4}{2}, \Routes{0}{1+}{2k+4}{2l+6}{1}}\\
        &= \impartialGameSet{\ast, \ast 4, \ast, \ast} \\
        &= 0\ \\
    \Routes{0}{1+}{2k+5}{2l+4}{2} 
        &= \impartialGameSet{\Routes{0}{1+}{2l+4}{2}{0}, \Routes{0}{1+}{2k+4}{2l+4}{2}, \Routes{0}{1+}{2k+6}{2l+3}{2}, \Routes{0}{1+}{2k+5}{2l+5}{1}}\\
        &= \impartialGameSet{0, \ast 2, 0, 0} \\
        &= \ast\ \\
    \Routes{0}{1+}{2k+5}{2l+5}{2} 
        &= \impartialGameSet{\Routes{0}{1+}{2l+5}{2}{0}, \Routes{0}{1+}{2k+4}{2l+5}{2}, \Routes{0}{1+}{2k+6}{2l+4}{2}, \Routes{0}{1+}{2k+5}{2l+6}{1}}\\
        &= \impartialGameSet{\ast, 0,  \ast 2, \ast 3} \\
        &= \ast 4\ \\
\end{align*}

Next, for 3 six-length paths as in Table \ref{table:Routes02463} $\Routes{0}{1+}{i}{j}{3}$:
\begin{align*}
    \Routes{0}{1+}{2k+4}{2l+4}{3} 
        &= \impartialGameSet{\Routes{0}{1+}{2l+4}{3}{0}, \Routes{0}{1+}{2k+3}{2l+4}{3}, \Routes{0}{1+}{2k+5}{2l+3}{3}, \Routes{0}{1+}{2k+4}{2l+5}{2}}\\
        &= \impartialGameSet{0, \ast 2, \ast 3, 0} \\
        &= \ast\  \\
    \Routes{0}{1+}{2k+4}{2l+5}{3} 
        &= \impartialGameSet{\Routes{0}{1+}{2l+5}{3}{0}, \Routes{0}{1+}{2k+3}{2l+5}{3}, \Routes{0}{1+}{2k+5}{2l+4}{3}, \Routes{0}{1+}{2k+4}{2l+6}{2}}\\
        &= \impartialGameSet{\ast 2, \ast 3, \ast 2, \ast 2} \\
        &= 0\  \\
    \Routes{0}{1+}{2k+5}{2l+4}{3} 
        &= \impartialGameSet{\Routes{0}{1+}{2l+4}{3}{0}, \Routes{0}{1+}{2k+4}{2l+4}{3}, \Routes{0}{1+}{2k+6}{2l+3}{3}, \Routes{0}{1+}{2k+5}{2l+5}{2}}\\
        &= \impartialGameSet{0, \ast, 0, \ast 4} \\
        &= \ast 2\  \\
    \Routes{0}{1+}{2k+5}{2l+5}{3} 
        &= \impartialGameSet{\Routes{0}{1+}{2l+5}{3}{0}, \Routes{0}{1+}{2k+4}{2l+5}{3}, \Routes{0}{1+}{2k+6}{2l+4}{3}, \Routes{0}{1+}{2k+5}{2l+6}{2}}\\
        &= \impartialGameSet{\ast 2, 0,  \ast, \ast} \\
        &= \ast 3\  \\
\end{align*}

Finally, in the case with 4 or more length-six paths also for Table \ref{table:Routes02463} $\Routes{0}{1+}{i}{j}{4+}$:
\begin{align*}
    \Routes{0}{1+}{2k+4}{2l+4}{4+} 
        &= \impartialGameSet{\Routes{0}{1+}{2l+4}{4+}{0}, \Routes{0}{1+}{2k+3}{2l+4}{4+}, \Routes{0}{1+}{2k+5}{2l+3}{4+}, \Routes{0}{1+}{2k+4}{2l+5}{3+}}\\
        &= \impartialGameSet{0, \ast 2, \ast 3, 0} \\
        &= \ast\  \\
    \Routes{0}{1+}{2k+4}{2l+5}{4+} 
        &= \impartialGameSet{\Routes{0}{1+}{2l+5}{4+}{0}, \Routes{0}{1+}{2k+3}{2l+5}{4+}, \Routes{0}{1+}{2k+5}{2l+4}{4+}, \Routes{0}{1+}{2k+4}{2l+6}{3+}}\\
        &= \impartialGameSet{\ast \text{ or } \ast 2, \ast 3, \ast 2, \ast } \\
        &= 0\  \\
    \Routes{0}{1+}{2k+5}{2l+4}{4+} 
        &= \impartialGameSet{\Routes{0}{1+}{2l+4}{4+}{0}, \Routes{0}{1+}{2k+4}{2l+4}{4+}, \Routes{0}{1+}{2k+6}{2l+3}{4+}, \Routes{0}{1+}{2k+5}{2l+5}{3+}}\\
        &= \impartialGameSet{0, \ast, 0, \ast 3} \\
        &= \ast 2\  \\
    \Routes{0}{1+}{2k+5}{2l+5}{4+} 
        &= \impartialGameSet{\Routes{0}{1+}{2l+5}{4+}{0}, \Routes{0}{1+}{2k+4}{2l+5}{4+}, \Routes{0}{1+}{2k+6}{2l+4}{4+}, \Routes{0}{1+}{2k+5}{2l+6}{3}}\\
        &= \impartialGameSet{\ast \text{ or } \ast 2, 0,  \ast, \ast 2} \\
        &= \ast 3\  \\
\end{align*}

For the next case, $\Routes{1}{1+}{i}{j}{0}$ (Table \ref{table:Routes135}) we can do this in one round: 
\begin{align*}
    \Routes{1}{1+}{2k+4}{2l+4}{0} 
        &= \impartialGameSet{\Routes{0}{1+}{2k+4}{2l+4}{0}, \Routes{1}{1+}{2k+3}{2l+4}{0}, \Routes{1}{1+}{2k+5}{2l+3}{0}}\\
        &= \impartialGameSet{0, 0, 0} \\
        &= \ast\  \\
    \Routes{1}{1+}{2k+4}{2l+5}{0} 
        &= \impartialGameSet{\Routes{0}{1+}{2k+4}{2l+5}{0}, \Routes{1}{1+}{2k+3}{2l+5}{0}, \Routes{1}{1+}{2k+5}{2l+4}{0}}\\
        &= \impartialGameSet{0, 0, 0 } \\
        &= \ast\  \\
    \Routes{1}{1+}{2k+5}{2l+4}{0} 
        &= \impartialGameSet{\Routes{0}{1+}{2k+5}{2l+4}{0}, \Routes{1}{1+}{2k+4}{2l+4}{0}, \Routes{1}{1+}{2k+6}{2l+3}{0}}\\
        &= \impartialGameSet{\ast, \ast, \ast} \\
        &= 0\  \\
    \Routes{1}{1+}{2k+5}{2l+5}{0} 
        &= \impartialGameSet{\Routes{0}{1+}{2k+5}{2l+5}{0}, \Routes{1}{1+}{2k+4}{2l+5}{0}, \Routes{1}{1+}{2k+6}{2l+4}{0}}\\
        &= \impartialGameSet{\ast 2, \ast, \ast} \\
        &= 0\  \\
\end{align*}

Third, with the $\Routes{0}{0}{i}{j}{k}$ case, we have to work with increasing numbers of length-six paths again.  We start with $\Routes{0}{0}{i}{j}{0}$, as in Table \ref{table:Routes24}:
\begin{align*}
    \Routes{0}{0}{2k+4}{2l+4}{0} 
        &= \impartialGameSet{\Routes{1}{1+}{2l+4}{0}{0}, \Routes{0}{1+}{2k+3}{2l+4}{0}, \Routes{0}{0}{2k+5}{2l+3}{0}}\\
        &= \impartialGameSet{\ast, \ast, \ast} \\
        &= 0\  \\
    \Routes{0}{0}{2k+4}{2l+5}{0} 
        &= \impartialGameSet{\Routes{1}{1+}{2l+5}{0}{0}, \Routes{0}{1+}{2k+3}{2l+5}{0}, \Routes{0}{0}{2k+5}{2l+4}{0}}\\
        &= \impartialGameSet{0, \ast 2, \ast 2} \\
        &= \ast\  \\
    \Routes{0}{0}{2k+5}{2l+4}{0} 
        &= \impartialGameSet{\Routes{1}{1+}{2l+4}{0}{0}, \Routes{0}{1+}{2k+4}{2l+4}{0}, \Routes{0}{0}{2k+6}{2l+3}{0}}\\
        &= \impartialGameSet{\ast, 0, \ast} \\
        &= \ast 2\  \\
    \Routes{0}{0}{2k+5}{2l+5}{0} 
        &= \impartialGameSet{\Routes{1}{1+}{2l+5}{0}{0}, \Routes{0}{1+}{2k+4}{2l+5}{0}, \Routes{0}{0}{2k+6}{2l+4}{0}}\\
        &= \impartialGameSet{0, 0,0 } \\
        &= \ast\ \ \\
\end{align*}

With 1 length-six path, as in Table \ref{table:Routes2461}, we have $\Routes{0}{0}{i}{j}{1}$:
\begin{align*}
    \Routes{0}{0}{2k+4}{2l+4}{1} 
        &= \impartialGameSet{\Routes{1}{1+}{2l+4}{1}{0}, \Routes{0}{1+}{2k+3}{2l+4}{1}, \Routes{0}{0}{2k+5}{2l+3}{1}, \Routes{0}{0}{2k+4}{2l+5}{0}}\\
        &= \impartialGameSet{\ast, \ast 3, \ast 3, \ast} \\
        &= 0\  \\
    \Routes{0}{0}{2k+4}{2l+5}{1} 
        &= \impartialGameSet{\Routes{1}{1+}{2l+5}{1}{0}, \Routes{0}{1+}{2k+3}{2l+5}{1}, \Routes{0}{0}{2k+5}{2l+4}{1}, \Routes{0}{0}{2k+4}{2l+6}{0}}\\
        &= \impartialGameSet{0, 0, 0, 0} \\
        &= \ast\  \\
    \Routes{0}{0}{2k+5}{2l+4}{1} 
        &= \impartialGameSet{\Routes{1}{1+}{2l+4}{1}{0}, \Routes{0}{1+}{2k+4}{2l+4}{1}, \Routes{0}{0}{2k+6}{2l+3}{1}, \Routes{0}{0}{2k+5}{2l+5}{0}}\\
        &= \impartialGameSet{\ast, \ast, \ast, \ast} \\
        &= \ast 0\  \\
    \Routes{0}{0}{2k+5}{2l+5}{1} 
        &= \impartialGameSet{\Routes{1}{1+}{2l+5}{1}{0}, \Routes{0}{1+}{2k+4}{2l+5}{1}, \Routes{0}{0}{2k+6}{2l+4}{1}, \Routes{0}{0}{2k+5}{2l+6}{0}}\\
        &= \impartialGameSet{0, \ast, 0, \ast 2 } \\
        &= \ast 3\  \\
\end{align*}

For 2 length-six paths in Table \ref{table:Routes2462}, $\Routes{0}{0}{i}{j}{2}$:
\begin{align*}
    \Routes{0}{0}{2k+4}{2l+4}{2} 
        &= \impartialGameSet{\Routes{1}{1+}{2l+4}{2}{0}, \Routes{0}{1+}{2k+3}{2l+4}{2}, \Routes{0}{0}{2k+5}{2l+3}{2}, \Routes{0}{0}{2k+4}{2l+5}{1}}\\
        &= \impartialGameSet{\ast, \ast , \ast , \ast} \\
        &= 0\  \\
    \Routes{0}{0}{2k+4}{2l+5}{2} 
        &= \impartialGameSet{\Routes{1}{1+}{2l+5}{2}{0}, \Routes{0}{1+}{2k+3}{2l+5}{2}, \Routes{0}{0}{2k+5}{2l+4}{2}, \Routes{0}{0}{2k+4}{2l+6}{1}}\\
        &= \impartialGameSet{0, \ast 4, 0, 0} \\
        &= \ast\  \\
    \Routes{0}{0}{2k+5}{2l+4}{2} 
        &= \impartialGameSet{\Routes{1}{1+}{2l+4}{2}{0}, \Routes{0}{1+}{2k+4}{2l+4}{2}, \Routes{0}{0}{2k+6}{2l+3}{2}, \Routes{0}{0}{2k+5}{2l+5}{1}}\\
        &= \impartialGameSet{\ast, \ast 2, \ast, \ast 3} \\
        &= \ast 0\  \\
    \Routes{0}{0}{2k+5}{2l+5}{2} 
        &= \impartialGameSet{\Routes{1}{1+}{2l+5}{2}{0}, \Routes{0}{1+}{2k+4}{2l+5}{2}, \Routes{0}{0}{2k+6}{2l+4}{2}, \Routes{0}{0}{2k+5}{2l+6}{1}}\\
        &= \impartialGameSet{0, 0, 0, 0} \\
        &= \ast\  \\
\end{align*}

For 3 length-six paths in Table \ref{table:Routes2463}, $\Routes{0}{0}{i}{j}{3}$:
\begin{align*}
    \Routes{0}{0}{2k+4}{2l+4}{3} 
        &= \impartialGameSet{\Routes{1}{1+}{2l+4}{3}{0}, \Routes{0}{1+}{2k+3}{2l+4}{3}, \Routes{0}{0}{2k+5}{2l+3}{3}, \Routes{0}{0}{2k+4}{2l+5}{2}}\\
        &= \impartialGameSet{\ast, \ast 2, \ast , \ast} \\
        &= 0\  \\
    \Routes{0}{0}{2k+4}{2l+5}{3} 
        &= \impartialGameSet{\Routes{1}{1+}{2l+5}{3}{0}, \Routes{0}{1+}{2k+3}{2l+5}{3}, \Routes{0}{0}{2k+5}{2l+4}{3}, \Routes{0}{0}{2k+4}{2l+6}{2}}\\
        &= \impartialGameSet{0, \ast 3, 0, 0} \\
        &= \ast\  \\
    \Routes{0}{0}{2k+5}{2l+4}{3} 
        &= \impartialGameSet{\Routes{1}{1+}{2l+4}{3}{0}, \Routes{0}{1+}{2k+4}{2l+4}{3}, \Routes{0}{0}{2k+6}{2l+3}{3}, \Routes{0}{0}{2k+5}{2l+5}{2}}\\
        &= \impartialGameSet{\ast, \ast, \ast, \ast} \\
        &= \ast 0\  \\
    \Routes{0}{0}{2k+5}{2l+5}{3} 
        &= \impartialGameSet{\Routes{1}{1+}{2l+5}{3}{0}, \Routes{0}{1+}{2k+4}{2l+5}{3}, \Routes{0}{0}{2k+6}{2l+4}{3}, \Routes{0}{0}{2k+5}{2l+6}{2}}\\
        &= \impartialGameSet{0, 0, 0, 0} \\
        &= \ast\  \\
\end{align*}

For 4 or more length-six paths, also from Table \ref{table:Routes2463}, $\Routes{0}{0}{i}{j}{4+}$:
\begin{align*}
    \Routes{0}{0}{2k+4}{2l+4}{4+} 
        &= \impartialGameSet{\Routes{1}{1+}{2l+4}{4+}{0}, \Routes{0}{1+}{2k+3}{2l+4}{4+}, \Routes{0}{0}{2k+5}{2l+3}{4+}, \Routes{0}{0}{2k+4}{2l+5}{3+}}\\
        &= \impartialGameSet{\ast, \ast 2, \ast , \ast} \\
        &= 0\  \\
    \Routes{0}{0}{2k+4}{2l+5}{4+} 
        &= \impartialGameSet{\Routes{1}{1+}{2l+5}{4+}{0}, \Routes{0}{1+}{2k+3}{2l+5}{4+}, \Routes{0}{0}{2k+5}{2l+4}{4+}, \Routes{0}{0}{2k+4}{2l+6}{3+}}\\
        &= \impartialGameSet{0, \ast 3, 0, 0} \\
        &= \ast\  \\
    \Routes{0}{0}{2k+5}{2l+4}{4+} 
        &= \impartialGameSet{\Routes{1}{1+}{2l+4}{4+}{0}, \Routes{0}{1+}{2k+4}{2l+4}{4+}, \Routes{0}{0}{2k+6}{2l+3}{4+}, \Routes{0}{0}{2k+5}{2l+5}{3+}}\\
        &= \impartialGameSet{\ast, \ast, \ast, \ast} \\
        &= \ast 0\  \\
    \Routes{0}{0}{2k+5}{2l+5}{4+} 
        &= \impartialGameSet{\Routes{1}{1+}{2l+5}{4+}{0}, \Routes{0}{1+}{2k+4}{2l+5}{4+}, \Routes{0}{0}{2k+6}{2l+4}{4+}, \Routes{0}{0}{2k+5}{2l+6}{3+}}\\
        &= \impartialGameSet{0, 0, 0, 0} \\
        &= \ast\  \\
\end{align*}

In the final case, $\Routes{1}{0}{i}{j}{0}$ as in Table \ref{table:Routes35}, we can again do this in one round: 
\begin{align*}
    \Routes{1}{0}{2k+4}{2l+4}{0} 
        &= \impartialGameSet{\Routes{0}{0}{2k+4}{2l+4}{0}, \Routes{1}{1+}{2k+3}{2l+4}{0}, \Routes{1}{0}{2k+5}{2l+3}{0}}\\
        &= \impartialGameSet{0, 0, 0} \\
        &= \ast\  \\
    \Routes{1}{0}{2k+4}{2l+5}{0} 
        &= \impartialGameSet{\Routes{0}{0}{2k+4}{2l+5}{0}, \Routes{1}{1+}{2k+3}{2l+5}{0}, \Routes{1}{0}{2k+5}{2l+4}{0}}\\
        &= \impartialGameSet{\ast , 0, 0 } \\
        &= \ast 2\  \\
    \Routes{1}{0}{2k+5}{2l+4}{0} 
        &= \impartialGameSet{\Routes{0}{0}{2k+5}{2l+4}{0}, \Routes{1}{1+}{2k+4}{2l+4}{0}, \Routes{1}{0}{2k+6}{2l+3}{0}}\\
        &= \impartialGameSet{\ast 2, \ast, \ast 2} \\
        &= 0\  \\
    \Routes{1}{1+}{2k+5}{2l+5}{0} 
        &= \impartialGameSet{\Routes{0}{0}{2k+5}{2l+5}{0}, \Routes{1}{1+}{2k+4}{2l+5}{0}, \Routes{1}{0}{2k+6}{2l+4}{0}}\\
        &= \impartialGameSet{\ast, \ast, \ast} \\
        &= 0\  \\
\end{align*}
\end{proof}

\end{document}